\documentclass[a4paper]{article}
\usepackage{comment}

\usepackage[dvips]{graphicx}
\usepackage{amssymb}
\usepackage{amsmath}
\usepackage{mathrsfs}
\usepackage{color}
\usepackage{amsthm}
\usepackage{multirow}
\DeclareMathAlphabet{\mathpzc}{OT1}{pzc}{m}{it}

\newtheorem{theorem}{Theorem}[section]
\newtheorem{corollary}[theorem]{Corollary}

\newtheorem{lemma}[theorem]{Lemma}
\newtheorem{proposition}[theorem]{Proposition}

\theoremstyle{definition}

\newtheorem{remark}[theorem]{Remark}

\newcommand{\epsi}{\varepsilon}

\newcommand{\vecf}{\mbox{\boldmath $ f $}}

\newcommand{\vecu}{\mbox{\boldmath $ u $}}
\newcommand{\vecv}{\mbox{\boldmath $ v $}}

\newcommand{\vecphi}{\mbox{\boldmath $ \phi $}}
\newcommand{\vecpsi}{\mbox{\boldmath $ \psi $}}
\newcommand{\vectau}{\mbox{\boldmath $ \tau $}}
\newcommand{\veckappa}{\mbox{\boldmath $ \kappa $}}

\newcommand{\scrD}{\mathscr{D}}
\newcommand{\scrM}{\mathscr{M}}
\newcommand{\scrG}{\mathscr{G}}
\newcommand{\scrH}{\mathscr{H}}
\newcommand{\scrF}{\mathscr{F}}
\newcommand{\scrN}{\mathscr{N}}
\newcommand{\scrK}{\mathscr{K}}
\newcommand{\scrP}{\mathscr{P}}
\newcommand{\scrQ}{\mathscr{Q}}
\newcommand{\scrR}{\mathscr{R}}
\newcommand{\scrS}{\mathscr{S}}

\newcommand{\calE}{\mathcal{E}}
\newcommand{\calL}{\mathcal{L}}
\newcommand{\calO}{\mathcal{O}}

\newcommand{\R}{\mathbb{R}}
\newcommand{\Z}{\mathbb{Z}}
\newcommand{\N}{\mathbb{N}}

\newcommand{\vphi}{\varphi}

\makeatletter
    
    \@addtoreset{equation}{section}
  \makeatother

\title{Variational formulae and estimates of O'Hara's knot energies}
\author{Shoya Kawakami\\Takeyuki Nagasawa}
\date{\today}

\begin{document}

\maketitle

%
%

\begin{abstract}
O'Hara's energies,
introduced by Jun O'Hara,
were proposed to answer the question of what is the canonical shape in a given knot type,
and were configured so that the less the energy value of a knot is,
the ``better" its shape is.
The existence and regularity of minimizers has been well studied.
In this article,
we calculate the first and second variational formulae of the
$(\alpha,p)$-O'Hara energies and show absolute integrability,
uniform boundedness,
and continuity properties.
Although several authors have already considered the variational formulae of the
$(\alpha,1)$-O'Hara energies,
their techniques do not seem to be applicable to the case
$p>1$.
We obtain the variational formulae in a novel manner by extracting a certain function from the energy density.
All of the
$(\alpha,p)$-energies are made from this function,
and by analyzing it,
we obtain not only the variational formulae but also estimates in several function spaces.
\end{abstract}

%
%

\section{Introduction}

In his papers \cite{O91,O92},
O'Hara proposed several energies for knots to determine the {\it canonical shape} of  a knot in a given knot class.
In order to describe these energies,
let
$\vecf : \R/\calL\Z \ni s \mapsto \vecf(s)$
be an arclength parametrization of a knot,
or more generally,
of a closed curve in
$\R^n$
without self-intersections.
For positive constants
$\alpha$
and
$p$,
the O'Hara
$(\alpha , p)$-energy is defined as
\[
	\calE_{(\alpha,p)}(\vecf)
	=
	\iint_{(\R/\calL\Z)^2}
	\scrM_{(\alpha,p)}(\vecf)(s_1,s_2) ds_1ds_2,
\]
where
\[
	\scrM_{(\alpha,p)}(\vecf)(s_1,s_2)
	=
	\left(
	\frac 1 { \| \vecf(s_1)-\vecf(s_2) \|_{\R^n}^\alpha } - \frac 1 { \scrD(\vecf(s_1) , \vecf(s_2))^\alpha }
	\right)^p .
\]
Here,
$\| \vecf(s_1)-\vecf(s_2) \|_{\R^n}$
and
$\scrD(\vecf(s_1),\vecf(s_2))$
are the extrinsic and the intrinsic distances between two points
$\vecf(s_1)$
and
$\vecf(s_2)$
on the curve,
respectively.

Amongst these energies,
the case
$p=1$
has been well-studied by many authors and here we describe several important contributions in this direction.
For example,
the first variational formula was derived in
\cite{FHW94,Rei12},
and the existence of a minimizer in a given knot class was studied in
\cite{FHW94,O94}.
In
\cite{ACFGH03},
it was shown that the right circle is the only global minimizer of the energy under length-constraint for appropriate
$\alpha$.
We also mention that there are several recent works concerning the gradient flow
\cite{B12-2,B18}
and the regularity of critical points
\cite{BRS16,BR13,BV19,FHW94,He00,Rei12}
of the
$(\alpha,1)$-energies.
The obvious difficulty in the analysis is that the energy density has a singularity on the diagonal set
$\{ (s,s) \,|\, s \in \R/\calL\Z \}$.
To avoid this,
the variational formulae were first derived using Cauchy's principal value \cite{FHW94,Rei12}.
The absolute integrability of the first and second variational formulae was shown by Ishizeki-Nagasawa \cite{IN15}
in the energy class.
They used the decomposition of the
$(2,1)$-energy
shown in
\cite{IN14},
which gives an expression of the energy density without using the intrinsic distance.
Furthermore, in
\cite{IN15},
they derived other estimates of variational formulae in several function spaces and a similar technique is applicable to other
$(\alpha,1)$-energies;
see \cite{INpre}.

By comparison,
the case
$p>1$
is rather less well studied.
Even the explicit expression of the first variational formula in the sense of Cauchy's principle value seems not to have been given.
In this paper,
the first and second variational formulae of
$(\alpha,p)$-energies will be given,
and several estimates will be shown:
absolute integrability,
uniform boundedness,
and continuity.
We,
however,
do not have a decomposition like the
$(\alpha,1)$
case,
and therefore the technique in \cite{IN15,INpre} cannot be used.
Instead,
we pay attention to the function
\[
	\scrN ( \vecu , \vecv ) =
	\frac 1 { 2 \| \vecf ( s_1 ) - \vecf ( s_2 ) \|_{ \R^n }^2 }
	\int_{ s_1 }^{ s_2 }
	\int_{ s_1 }^{ s_2 }
	( \vecu ( s_3 ) - \vecu ( s_4 ) )
	\cdot
	( \vecv ( s_3 ) - \vecv ( s_4 ) )
	ds_3 ds_4.
\]
By use of
$\scrN$,
the
$(\alpha,p)$-energy may be written as
\[
	\calE_{ ( \alpha , p ) } ( \vecf )
	=
	\iint_{ ( \R / \calL \Z )^2 }
	\left(
	\frac { \vphi_\alpha \left( \scrN ( \vecf^\prime , \vecf^\prime ) \right) }
	{ \| \vecf ( s_1 ) - \vecf ( s_2 ) \|_{ \R^n }^\alpha }
	\right)^p
	ds_1 ds_2,
\]
where
\[
	\vphi_\alpha (t)
	=
	1 - \frac 1 { ( 1 + t )^{ \frac \alpha 2 } }.
\]
Here,
we note that if
$ \calE_{ ( \alpha , p ) } ( \vecf ) < \infty $,
then
$ \vecf^\prime $
exists almost everywhere;
we refer the reader to
\cite{B12}
for this fact.
Since
$ \scrN ( \vecf^\prime , \vecf^\prime ) $
is non-negative,
and since
$ \vphi_\alpha $
is smooth for
$ t \geq 0 $,
the derivation and estimation of variational formulae mainly reduce to the study of
$ \scrN $.

We are now in a position to describe our main result of this article.
For the simplicity of notation,
we will use
$\Delta_j^i \vecu$
to mean
$\vecu(s_i)-\vecu(s_j)$
for a function
$\vecu : \R/\calL\Z \to \R^d$,
where
$d=1$
or
$n$,
and
$\Delta_j^i s$
to mean
$s_i-s_j$
for
$s_i$,
$s_j \in \R/\calL\Z$.
Furthermore,
we rewrite
$\Delta_2^1 \vecu$
and
$\scrD(\vecf(s_1),\vecf(s_2))$
as
$\Delta \vecu$
and
$\scrD(\vecf)$,
respectively.
Also,
we let
$\vectau = \vecf^\prime$
be the unit tangent vector.

For a geometric quantity
$\scrF(\vecf)$
of the closed curve
$\vecf$,
and for functions
$\vecphi$,
$\vecpsi$,
from
$\R/\calL\Z$
to
$\R^n$,
let
$\delta$
and
$\delta^2$
be given by
\begin{align*}
	\delta \scrF(\vecf) [\vecphi]
	&=
	\left. \frac d {d\epsi} \scrF(\vecf + \epsi\vecphi) \right|_{\epsi =0},
\\
	\delta^2 \scrF(\vecf) [\vecphi,\vecpsi]
	&=
	\left.
	\frac{\partial^2}{\partial\epsi_1 \partial\epsi_2} \scrF(\vecf + \epsi_1\vecphi + \epsi_2\vecpsi)
	\right|_{\epsi_1=\epsi_2=0}.
\end{align*}
The first and second variational formulae,
$\scrG_{(\alpha,p)}(\vecf)[\vecphi]$
and
$\scrH_{(\alpha,p)}(\vecf)[\vecphi,\vecpsi]$,
are given by
\begin{align*}
	\scrG_{(\alpha,p)}(\vecf)[\vecphi] \, ds_1 ds_2
	&=
	\delta ( \scrM_{(\alpha,p)}(\vecf) \, ds_1 ds_2 )[\vecphi],
\\
	\scrH_{(\alpha,p)}(\vecf)[\vecphi,\vecpsi]ds_1 ds_2
	&=
	\delta^2 ( \scrM_{(\alpha,p)}(\vecf) \, ds_1 ds_2 )[\vecphi,\vecpsi].
\end{align*}
The purpose of this article is to give certain new expressions and estimates for these variational formulae.
The expression will be given in \S~2.
In \S~3,
we prove the
$ L^1 $,
$ L^\infty $
and
$ C^0 $-estimates for them on the appropriate function spaces.
The first space is the \textit{Sobolev-Slobodeckij space}
$W^{k+\sigma,q}(\R/\calL\Z,\R^n)$
given by
\[
	W^{k+\sigma , q}(\R/\calL\Z,\R^n)
	=
	\{
	\vecf \in W^{k,q}(\R/\calL\Z,\R^n)
	\,|\,
	[\vecf^{(k)}]_{W^{\sigma,q}}
	< \infty
	\}
\]
for
$k \in \{ 0 \} \cup \N$,
$0<\sigma<1$,
and
$q \geq 1$.
We equip
$W^{k+\sigma,q}(\R/\calL\Z,\R^n)$
with the norm
\[
	\| \vecf \|_{W^{k+\sigma , q}}
	=
	\left(
	\| \vecf \|_{W^{k,q}}^q + [\vecf^{(k)}]_{W^{\sigma,q}}^q
	\right)^{\frac 1 q},
\]
where
\[
	[\vecf^{(k)}]_{W^{\sigma,q}}
	=
	\left(
	\iint_{(\R/\calL\Z)^2}
	\frac{ \| \Delta \vecf^{(k)} \|_{\R^n}^q }{ |\Delta s|^{1+\sigma q} }
	ds_1ds_2
	\right)^{\frac 1 q}.
\]
For
$\alpha \in (0,\infty)$,
$p \in [1,\infty)$
such that
$2 \leq \alpha p < 2p+1$,
and
$\sigma = (\alpha p-1)/(2p)$,
it was shown in
\cite{B12}
that
$\calE_{(\alpha,p)}(\vecf) < \infty$
if and only if
$\vecf$
belongs to
$W^{1+\sigma , 2p}(\R/\calL\Z , \R^n) \cap W^{1,\infty}(\R/\calL\Z,\R^n)$
and is bi-Lipschitz.
We will show the
$L^1$-estimate for the variational formulae on the same space.

By use of an appropriate weight,
the
$ L^\infty $-estimate holds on some H\"{o}lder or Lipschitz space
$ C^{k,\beta}(\R/\calL\Z,\R^n)$
with
$ k \in \mathbb{N} \cup \{ 0 \} $
and
$ \beta \in ( 0,1 ] $.
Here,
\[
	C^{k,\beta}(\R/\calL\Z,\R^n)
	=
	\{
	\vecf \in C^k(\R/\calL\Z,\R^n)
	\,|\,
	[\vecf^{(k)}]_{C^{0,\beta}}
	< \infty
	\}
\]
and this space is equipped with the norm
\[
	\| \vecf \|_{C^{k,\beta}}
	=
	\| \vecf \|_{C^k} + [\vecf^{(k)}]_{C^{0,\beta}},
\]
where
\[
	[\vecf^{(k)}]_{C^{0,\beta}}
	=
	\sup_{s_1,s_2 \in \R/\calL\Z}
	\frac{ \| \Delta \vecf^{(k)} \|_{\R^n} }{ |\Delta s|^\beta }.
\]
Note that we will see later that the weight is necessary.

The completion of
$ C^\infty (\R/\calL\Z,\R^n) $
in the H\"{o}lder space is known as the \textit{little H\"{o}lder space}
$h^{k,\beta}(\R/\calL\Z,\R^n)$
when
$ \beta < 1 $.
The completion of
$ C^\infty (\R/\calL\Z,\R^n) $
in
$ C^{k,1} (\R/\calL\Z,\R^n)$
is
$ C^{ k+1 } (\R/\calL\Z,\R^n) $.
We have the continuity of the energy density,
and the first and second variational formulae in the completion space.
The precise statement is as follows.
\begin{theorem}\label{mainthm}
Let
$\alpha \in (0,\infty)$,
$p \in [1,\infty)$
satisfy
$2 \leq \alpha p < 2p+1$,
and set
$\sigma = (\alpha p-1)/(2p)$.
Assume that
$\vecf$
is bi-Lipschitz,
i.e.,
there exists a positive constant
$C_{\mathrm{b}} >0$
such that
$\scrD(\vecf) \leq C_{\mathrm{b}} \| \Delta \vecf \|_{\R^n}$.
\begin{enumerate}\renewcommand{\labelenumi}{\upshape\textrm{\theenumi.}}
\item
If
$\vecf$,
$\vecphi$,
$\vecpsi \in W^{1+\sigma , 2p}(\R/\calL\Z , \R^n) \cap W^{1,\infty}(\R/\calL\Z , \R^n)$,
then
$\scrM_{(\alpha,p)} (\vecf)$,
$\scrG_{(\alpha,p)} (\vecf)[\vecphi]$,
$\scrH_{(\alpha,p)} (\vecf)[\vecphi,\vecpsi] \in L^1((\R/\calL\Z)^2)$.
Moreover,
there exists a positive constant
$C$
depending on
$\| \vectau \|_{W^{\sigma , 2p} \cap L^\infty}$,
$C_{\mathrm{b}}$,
$\alpha$,
and
$p$
such that
\begin{align*}
	\| \scrM_{(\alpha,p)} (\vecf) \|_{L^1((\R/\calL\Z)^2)} &\leq C,
\\
	\| \scrG_{(\alpha,p)} (\vecf)[\vecphi] \|_{L^1((\R/\calL\Z)^2)}
	&\leq C \| \vecphi^\prime \|_{W^{\sigma , 2p} \cap L^\infty},
\\
	\| \scrH_{(\alpha,p)} (\vecf)[\vecphi,\vecpsi] \|_{L^1((\R/\calL\Z)^2)}
	&\leq
	C \| \vecphi^\prime \|_{W^{\sigma , 2p} \cap L^\infty}
	\| \vecpsi^\prime \|_{W^{\sigma , 2p} \cap L^\infty}.
\end{align*}
\item
Let
$\beta \in (0,1]$.
If
$\vecf$,
$\vecphi$,
$\vecpsi \in C^{1,\beta}(\R/\calL\Z , \R^n)$,
then
$\scrD(\vecf)^{(\alpha-2\beta)p} \scrM_{(\alpha,p)} (\vecf)$,
$\scrD(\vecf)^{(\alpha-2\beta)p} \scrG_{(\alpha,p)} (\vecf)[\vecphi]$,
$\scrD(\vecf)^{(\alpha-2\beta)p} \scrH_{(\alpha,p)} (\vecf)[\vecphi,\vecpsi] \in L^\infty((\R/\calL\Z)^2)$.
Moreover,
there exists a positive constant
$C$
depending on
$\| \vectau \|_{C^{0,\beta}}$,
$C_{\mathrm{b}}$,
$\alpha$,
and
$p$
such that
\begin{align*}
	\| \scrD(\vecf)^{(\alpha-2\beta)p} \scrM_{(\alpha,p)} (\vecf) \|_{L^\infty((\R/\calL\Z)^2)}
	&\leq C,
\\
	\| \scrD(\vecf)^{(\alpha-2\beta)p} \scrG_{(\alpha,p)} (\vecf)[\vecphi] \|_{L^\infty((\R/\calL\Z)^2)}
	&\leq C \| \vecphi^\prime \|_{C^{0,\beta}},
\\
	\| \scrD(\vecf)^{(\alpha-2\beta)p}
	\scrH_{(\alpha,p)} (\vecf)[\vecphi,\vecpsi] \|_{L^\infty((\R/\calL\Z)^2)}
	&\leq
	C \| \vecphi^\prime \|_{C^{0,\beta}} \| \vecpsi^\prime \|_{C^{0,\beta}}.
\end{align*}
\item
For
$\beta \in (0,1]$,
let
\[
	X^\beta(\R/\calL\Z,\R^n) = \left\{
	\begin{array}{ll}
		h^{0,\beta}(\R/\calL\Z , \R^n) & \mbox{for}\ 0<\beta<1,
	\vspace{3pt}\\
		C^1(\R/\calL\Z , \R^n) & \mbox{for}\ \beta =1.
	\end{array}
	\right.
\]
If
$\vectau$,
$\vecphi^\prime$,
$\vecpsi^\prime \in X^\beta(\R/\calL\Z , \R^n)$,
then
$\scrD(\vecf)^{(\alpha-2\beta)p} \scrM_{(\alpha,p)} (\vecf)$,
$\scrD(\vecf)^{(\alpha-2\beta)p} \scrG_{(\alpha,p)} (\vecf)[\vecphi]$,
$\scrD(\vecf)^{(\alpha-2\beta)p} \scrH_{(\alpha,p)} (\vecf)[\vecphi,\vecpsi]$
can be extended to the diagonal set
$\{ (s,s) \,|\, s \in \R/\calL\Z \}$
such that these functions are continuous everywhere on
$(\R/\calL\Z)^2$,
and they satisfy the estimates
\begin{align*}
	\| \scrD(\vecf)^{(\alpha-2\beta)p} \scrM_{(\alpha,p)} (\vecf) \|_{C^0((\R/\calL\Z)^2)}
	&\leq
	C,
\\
	\| \scrD(\vecf)^{(\alpha-2\beta)p} \scrG_{(\alpha,p)} (\vecf)[\vecphi] \|_{C^0((\R/\calL\Z)^2)}
	&\leq
	C \| \vecphi^\prime \|_{X^\beta},
\\
	\| \scrD(\vecf)^{(\alpha-2\beta)p} \scrH_{(\alpha,p)} (\vecf)[\vecphi,\vecpsi] \|_{C^0((\R/\calL\Z)^2)}
	&\leq
	C \| \vecphi^\prime \|_{X^\beta} \| \vecpsi^\prime \|_{X^\beta}
\end{align*}
for some positive constant
$C$
depending on
$\| \vectau \|_{X^\beta}$,
$C_{\mathrm{b}}$,
$\alpha$,
and
$p$.
The limit functions
\begin{gather*}
	\lim_{(s_1,s_2) \to (s,s)}
	\scrD(\vecf)^{(\alpha-2\beta)p} \scrM_{(\alpha,p)} (\vecf),
\\
	\lim_{(s_1,s_2) \to (s,s)}
	\scrD(\vecf)^{(\alpha-2\beta)p} \scrG_{(\alpha,p)} (\vecf)[\vecphi],
\\
	\lim_{(s_1,s_2) \to (s,s)}
	\scrD(\vecf)^{(\alpha-2\beta)p} \scrH_{(\alpha,p)} (\vecf)[\vecphi,\vecpsi]
\end{gather*}
exist and are finite.
These vanish everywhere on
$\R/\calL\Z$
when
$\beta \in (0,1)$.
\end{enumerate}
\end{theorem}

The second assertion gives us the
$L^\infty$-estimates for
$\scrM_{(\alpha,p)}$,
$\scrG_{(\alpha,p)}$,
and
$\scrH_{(\alpha,p)}$ without the weight,
when
$\alpha \leq 2$.

\begin{corollary}\label{mainthm coro}
Let
$\alpha \in (0,\infty)$,
$p \in [1,\infty)$
satisfy
$2/p \leq \alpha \leq 2$.
If
$\vecf$,
$\vecphi$,
$\vecpsi \in C^{1,\frac \alpha 2}(\R/\calL\Z , \R^n)$
and
$\vecf$
is bi-Lipschitz,
then
$\scrM_{(\alpha,p)} (\vecf)$,
$\scrG_{(\alpha,p)} (\vecf)[\vecphi]$,
and
$\scrH_{(\alpha,p)} (\vecf)[\vecphi,\vecpsi]$
belong to
$L^\infty((\R/\calL\Z)^2)$.
\end{corollary}

The corresponding estimates to our main result were shown for the spacial case
$ \calE_{ (2,1) } $
in
\cite{IN15},
and in this sense our result is an extension to a wider class of O'Hara's energies.
\begin{remark}
When
$\alpha >2$,
we need the weight
$ \scrD ( \vecf )^{ ( \alpha - 2 \beta ) p } $
to obtain the uniform boundedness even if
$ \vecf $
is analytic.
For example,
let us consider the right circle with the total length
$ 2 \pi $,
i.e.,
\[
	\vecf(s) = (\cos s , \sin s , 0, \ldots ,0)
	\quad
	\mbox{for}\ s \in \R/2\pi\Z.
\]
Since
$\| \Delta \vecf \|_{\R^n}^2 = 2(1-\cos \Delta s)$,
it follows that
\begin{equation}\label{eq remark}
	\scrM_{(\alpha,p)}(\vecf)
	=
	\frac 1 { \| \Delta \vecf \|_{\R^n}^{\alpha p} }
	\left[
	1-
	\left\{
	\frac{ 2(1-\cos\Delta s) }{ (\Delta s)^2 }
	\right\}^{\frac \alpha 2}
	\right]^p.
\end{equation}
Using Taylor's theorem
\[
	2(1-\cos x)
	=
	x^2 - \frac 1 {12} x^4 + \calO(x^6)
	\quad
	\text{as}\ x \to 0,
\]
and thus we have
\[
	1-
	\left\{
	\frac{ 2(1-\cos\Delta s) }{ (\Delta s)^2 }
	\right\}^{\frac \alpha 2}
	=
	\frac \alpha {24} (\Delta s)^2
	+
	\calO((\Delta s)^4)
	\quad
	\text{as}\ \Delta s \to 0.
\]
Hence,
the right-hand side of
\eqref{eq remark}
is equal to
\[
	\left(
	\frac{ |\Delta s| }{ \| \Delta \vecf \|_{\R^n} }
	\right)^{\alpha p}
	\left(
	\frac \alpha {24} |\Delta s|^{2-\alpha}
	+
	\calO(|\Delta s|^{4-\alpha})
	\right)^p
	\quad
	\text{as}\ \Delta s \to 0.
\]
Therefore,
when
$\alpha >2$,
$ \scrD ( \vecf )^\gamma \scrM_{(\alpha,p)}(\vecf) $
is uniformly bounded if and only if
$ \gamma \geq ( \alpha - 2 ) p $.
\end{remark}

\paragraph*{Acknowledgments}
The authors would like to thank Professor Richard Neal Bez for English language editing and mathematical comments.

The second author is supported by Grant-in-Aid for Scientific Research (C) (No.17K05310),
Japan Society for the Promotion of Science.

%
%

\section{Variational formulae}\label{Sect. variational formulae}

We begin by recalling the definition of the function
$\scrN$.
For functions
$\vecu$,
$\vecv : \R/\calL\Z \to \R^d$,
where
$d=1$
or
$n$,
we set
\[
	\scrN(\vecu,\vecv)
	=
	\frac 1 { 2 \| \Delta \vecf \|_{\R^n}^2 }
	\int_{s_2}^{s_1} \int_{s_2}^{s_1}
	\Delta_4^3 \vecu \cdot \Delta_4^3 \vecv
	ds_3ds_4.
\]
We use
$\scrN(\vecu)$
instead of
$\scrN(\vecu,\vecu)$
for simplicity.

Setting
\begin{equation}\label{def Malpha}
	\scrM_\alpha(\vecf)
	=
	\frac{ \vphi_\alpha(\scrN(\vectau)) }{ \| \Delta \vecf \|_{\R^n}^\alpha },
\end{equation}
then the energy can be written as
\[
	\calE_{(\alpha,p)}(\vecf)
	=
	\iint_{(\R/\calL\Z)^2} ( \scrM_\alpha(\vecf) )^p ds_1ds_2.
\]
The first variational formula
$\scrG_{(\alpha,p)}$
can be derived as
\[
	\scrG_{(\alpha,p)}(\vecf)[\vecphi] ds_1ds_2
	=
	\delta \{ ( \scrM_\alpha(\vecf) )^p \} [\vecphi] ds_1ds_2
	+
	\scrM_\alpha(\vecf) \delta (ds_1ds_2)[\vecphi],
\]
and
\[
	\delta \{ (\scrM_\alpha(\vecf))^p \} [\vecphi]
	=
	p (\scrM_\alpha(\vecf))^{p-1} \delta \scrM_\alpha(\vecf) [\vecphi].
\]
The second variational formula
$\scrH_{(\alpha,p)}$
is calculated similarly.
Since
$\delta(ds_j) [\vecphi] = \vectau(s_j) \cdot \vecphi^\prime(s_j) ds_j$
holds
(see,
for example,
\cite{IN15}),
we obtain the following.
\begin{theorem}\label{variation}
$\scrG_{(\alpha,p)}(\vecf)[\vecphi]$
and
$\scrH_{(\alpha,p)}(\vecf)[\vecphi,\vecpsi]$
can be written as
\[
	\scrG_{(\alpha,p)}(\vecf)[\vecphi]
	=
	\sum_{i=1}^2 \scrG_i(\vecf)[\vecphi],
	\quad
	\scrH_{(\alpha,p)}(\vecf)[\vecphi,\vecpsi]
	=
	\sum_{i=1}^6 \scrH_i(\vecf)[\vecphi,\vecpsi],
\]
where
\begin{align*}
	\scrG_1(\vecf)[\vecphi]
	&=
	p(\scrM_\alpha(\vecf))^{p-1} \delta \scrM_\alpha(\vecf)[\vecphi],
\\
	\scrG_2(\vecf)[\vecphi]
	&=
	( \scrM_\alpha(\vecf) )^p
	( \vectau(s_1) \cdot \vecphi^\prime(s_1) + \vectau(s_2) \cdot \vecphi^\prime(s_2) ),
\\
	\scrH_1(\vecf)[\vecphi,\vecpsi]
	&=
	p(\scrM_\alpha(\vecf))^{p-1} \delta^2 \scrM_\alpha(\vecf) [\vecphi,\vecpsi],
\\
	\scrH_2(\vecf)[\vecphi,\vecpsi]
	&=
	p(p-1) ( \scrM_\alpha(\vecf) )^{p-2}
	\delta \scrM_\alpha(\vecf) [\vecphi] \delta \scrM_\alpha(\vecf) [\vecpsi],
\\
	\scrH_3(\vecf)[\vecphi,\vecpsi]
	&=
	\scrG_1(\vecf) [\vecphi]
	( \vectau(s_1) \cdot \vecpsi^\prime(s_1) + \vectau(s_2) \cdot \vecpsi^\prime(s_2) ),
\\
	\scrH_4(\vecf)[\vecphi,\vecpsi]
	&=
	\scrG_1(\vecf) [\vecpsi]
	( \vectau(s_1) \cdot \vecphi^\prime(s_1) + \vectau(s_2) \cdot \vecphi^\prime(s_2) ),
\\
	\scrH_5(\vecf)[\vecphi,\vecpsi]
	&=
	( \scrM_\alpha(\vecf) )^p
	\{
	\vecphi^\prime(s_1) \cdot \vecpsi^\prime(s_1) + \vecphi^\prime(s_2) \cdot \vecpsi^\prime(s_2)
\\
	&\quad
	\left.
	-2 ( \vectau(s_1) \cdot \vecphi^\prime(s_1) )( \vectau(s_1) \cdot \vecpsi^\prime(s_1) )
	-2 ( \vectau(s_2) \cdot \vecphi^\prime(s_2) )( \vectau(s_2) \cdot \vecpsi^\prime(s_2) )
	\right\},
\\
	\scrH_6(\vecf)[\vecphi,\vecpsi]
	&=
	( \scrM_\alpha(\vecf) )^p
	( \vectau(s_1) \cdot \vecphi^\prime(s_1) + \vectau(s_2) \cdot \vecphi^\prime(s_2) )
	( \vectau(s_1) \cdot \vecpsi^\prime(s_1) + \vectau(s_2) \cdot \vecpsi^\prime(s_2) ).
\end{align*}
\end{theorem}

In the remainder of this section,
we will give the exact expression of
$\delta \scrM_\alpha(\vecf)[\vecphi]$
and
$\delta^2 \scrM_\alpha(\vecf)[\vecphi,\vecpsi]$.
First we note that,
by \eqref{def Malpha},
we have
\begin{equation}
	\delta \scrM_\alpha(\vecf)[\vecphi]
	=
	\vphi_\alpha^\prime(\scrN(\vectau))
	\frac{ \delta \scrN(\vectau) [\vecphi] }{ \| \Delta \vecf \|_{\R^n}^\alpha }
	-
	\frac \alpha 2
	\frac{ \vphi_\alpha(\scrN(\vectau)) \delta \| \Delta \vecf \|_{\R^n}^2 [\vecphi] }
	{ \| \Delta \vecf \|_{\R^n}^{\alpha+2} },
\label{delta Malpha}
\end{equation}
and therefore we need variational formulae for
$\scrN(\vectau)$,
and
$\| \Delta \vecf \|_{\R^n}^2$.
It follows from the definition of
$\scrN(\vectau)$
that 
$\delta \scrN(\vectau)$
can be written in terms of
$\delta \vectau$
as well as
$\delta \| \Delta \vecf \|_{\R^n}^2$.
We can derive the second variational formula from the corresponding ingredients.

Firstly,
we show the ingredients as Lemma \ref{variation lemma}.
The variational formulae
$\delta \scrN(\vectau)$
and
$\delta^2 \scrN(\vectau)$
will be given in the forthcoming Lemma \ref{variation N}.
Also,
we give representations of
$\delta \scrM_\alpha(\vecf)$
and
$\delta^2 \scrM_\alpha(\vecf)$
in Proposition \ref{variation Malpha}.

For
$\vecu$,
$\vecv : \R/\calL\Z \to \R^n$,
let
\[
	\scrK(\vecu,\vecv)
	=
	\frac{ \Delta \vecu \cdot \Delta \vecv }{ \| \Delta \vecf \|_{\R^n}^2 }.
\]

\begin{lemma}\label{variation lemma}
The following variational formulae hold.
\begin{enumerate}\renewcommand{\labelenumi}{\upshape\textrm{\theenumi.}}
\item
$\delta \vectau(s_j)[\vecphi] = \vecphi^\prime(s_j) - (\vectau(s_j) \cdot \vecphi^\prime(s_j)) \vectau(s_j)$.
\item
$
	\delta \| \Delta^i_j \vectau \|_{\R^n}^2[\vecphi]
	=
	2\Delta^i_j \vectau \cdot \Delta^i_j \vecphi^\prime
	-
	\| \Delta^i_j \vectau \|_{\R^n}^2
	( \vectau(s_i) \cdot \vecphi^\prime(s_i) + \vectau(s_j) \cdot \vecphi^\prime(s_j) )
$.
\item
$
	\displaystyle
	\delta \| \Delta \vecf \|_{\R^n}^2 [\vecphi]
	=
	2 \scrK(\vecf,\vecphi) \| \Delta \vecf \|_{\R^n}^2
$.
\item
$
	\displaystyle
	\delta \left( \frac 1 { \| \Delta \vecf \|_{\R^n}^2 } \right)[\vecphi]
	=
	-2
	\frac{ \scrK(\vecf,\vecphi) }{ \| \Delta \vecf \|_{\R^n}^2 }
$.
\item
$
	\displaystyle
	\delta^2 \| \Delta \vecf \|_{\R^n}^2[\vecphi,\vecpsi]
	=
	2 \scrK(\vecphi,\vecpsi) \| \Delta \vecf \|_{\R^n}^2
$.
\end{enumerate}
\end{lemma}

\begin{proof}
See
\cite[Lemma 1]{IN15}.
\end{proof}

Before proving Lemma \ref{variation N},
we establish the following sublemma.

\begin{lemma}\label{variation N sublemma}
\begin{enumerate}\renewcommand{\labelenumi}{\upshape\textrm{\theenumi.}}
\item
Assume that
$\vecu : \R/\calL\Z \to \R^n$
is
$\vectau$
or
$\vecphi^\prime$.
Then,
$\delta \scrN(\vectau,\vecu)[\vecpsi]$
can be written as
\begin{align*}
	\delta \scrN(\vectau,\vecu) [\vecpsi]
	&=
	-2 \scrK(\vecf,\vecpsi) \scrN(\vectau,\vecu)
\\
	&\quad
	+
	\frac 1 { 2 \| \Delta \vecf \|_{\R^n} }
	\int_{s_2}^{s_1} \int_{s_2}^{s_1}
	\{
	\delta (\Delta_4^3 \vectau \cdot \Delta_4^3 \vecu)[\vecpsi]
\\
	&\qquad
	+
	( \Delta_4^3 \vectau \cdot \Delta_4^3 \vecu )
	( \vectau(s_3) \cdot \vecpsi^\prime(s_3) + \vectau(s_4) \cdot \vecpsi^\prime(s_4) )
	\}
	ds_3ds_4,
\end{align*}
where
\begin{multline*}
	\delta (\Delta_4^3 \vectau \cdot \Delta_4^3 \vecu)[\vecpsi]
\\
	=
	\left\{
	\begin{array}{ll}
		2\Delta_4^3 \vectau \cdot \Delta_4^3 \vecphi^\prime
		-
		\| \Delta_4^3 \vectau \|_{\R^n}^2
		( \vectau(s_3) \cdot \vecphi^\prime(s_3) + \vectau(s_4) \cdot \vecphi^\prime(s_4) ),
		& \vecu = \vectau,
	\vspace{5pt}\\
	\begin{aligned}
		\Delta_4^3 \vecphi^\prime \cdot \Delta_4^3 \vecpsi^\prime
		&- ( \Delta_4^3 \vectau \cdot \Delta_4^3 \vecphi^\prime )
		( \vectau(s_3) \cdot \vecpsi^\prime(s_3) + \vectau(s_4) \cdot \vecpsi^\prime(s_4) )
	\\
		&- \{ \Delta_4^3 ( \vectau \cdot \vecphi^\prime ) \}
		\{ \Delta_4^3 ( \vectau \cdot \vecpsi^\prime ) \},
	\end{aligned}
		& \vecu = \vecphi^\prime.
	\end{array}
	\right.
\end{multline*}
\item
It holds that
\[
	\delta \scrK(\vecf,\vecphi)[\vecpsi]
	=
	\scrK(\vecphi,\vecpsi)
	-
	2 \scrK(\vecf,\vecphi) \scrK(\vecf,\vecpsi).
\]
\end{enumerate}
\end{lemma}

\begin{proof}
\begin{enumerate}
\item By Lemma \ref{variation lemma},
we have
\begin{align*}
	\delta \scrN(\vectau,\vecu)[\vecpsi]
	&=
	\delta
	\left(
	\frac 1 { 2\| \Delta \vecf \|_{\R^n}^2 }
	\int_{s_2}^{s_1} \int_{s_2}^{s_1}
	\Delta_4^3 \vectau \cdot \Delta_4^3 \vecu ds_3ds_4
	\right) [\vecpsi]
\\
	&=
	\frac 1 2
	\delta
	\left(
	\frac 1 { \| \Delta \vecf \|_{\R^n}^2 }
	\right) [\vecpsi]
	\int_{s_2}^{s_1} \int_{s_2}^{s_1}
	\Delta_4^3 \vectau \cdot \Delta_4^3 \vecu ds_3ds_4
\\
	&\quad
	+
	\frac 1 { 2\| \Delta \vecf \|_{\R^n}^2 }
	\int_{s_2}^{s_1} \int_{s_2}^{s_1}
	\delta (\Delta_4^3 \vectau \cdot \Delta_4^3 \vecu) [\vecpsi] ds_3ds_4
\\
	&\quad
	+
	\frac 1 { 2\| \Delta \vecf \|_{\R^n}^2 }
	\int_{s_2}^{s_1} \int_{s_2}^{s_1} ( \Delta_4^3 \vectau \cdot \Delta_4^3 \vecu )
	\delta(ds_3ds_4)[\vecpsi]
\\
	&=
	-2\scrK(\vecf,\vecpsi) \scrN(\vectau,\vecu)
	+
	\frac 1 { 2\| \Delta \vecf \|_{\R^n}^2 }
	\int_{s_2}^{s_1} \int_{s_2}^{s_1}
	\{
	\delta (\Delta_4^3 \vectau \cdot \Delta_4^3 \vecu) [\vecphi]
\\
	&\qquad
	+
	( \Delta_4^3 \vectau \cdot \Delta_4^3 \vecu )
	(\vectau(s_3) \cdot \vecpsi^\prime(s_3) + \vectau(s_4) \cdot \vecpsi^\prime(s_4))
	\}
	ds_3ds_4.
\end{align*}
It remains to calculate
$\delta (\Delta_4^3 \vectau \cdot \Delta_4^3 \vecu) [\vecphi]$.
Since the case where
$\vecu = \vectau$
may be handled by Lemma \ref{variation lemma},
we deal with the case where
$\vecu = \vecphi^\prime$.
Then,
it follows from Lemma \ref{variation lemma} and
$\delta \vecphi^\prime(s_j)[\vecpsi] = - (\vectau(s_j) \cdot \vecpsi^\prime(s_j)) \vecphi^\prime(s_j)$
that
\begin{align*}
	&
	\delta ( \Delta_4^3 \vectau \cdot \Delta_4^3 \vecphi^\prime ) [\vecpsi]
\\
	&=
	\delta ( \Delta_4^3 \vectau ) [\vecpsi] \cdot \Delta_4^3 \vecphi^\prime
	+
	\Delta_4^3 \vectau \cdot \delta ( \Delta_4^3 \vecphi^\prime ) [\vecpsi]
\\
	&=
	\Delta_4^3 \vecpsi^\prime \cdot \Delta_4^3 \vecphi^\prime
	-
	\Delta_4^3 \{ ( \vectau \cdot \vecpsi^\prime ) \vectau \} \cdot \Delta_4^3 \vecphi^\prime
	-
	\Delta_4^3 \vectau \cdot \Delta_4^3 \{ ( \vectau \cdot \vecpsi^\prime ) \vecphi^\prime \}
\\
	&=
	\Delta_4^3 \vecphi^\prime \cdot \Delta_4^3 \vecpsi^\prime
	-
	[ ( \vectau(s_3) \cdot \vecpsi^\prime(s_3) ) \Delta_4^3 \vectau
	+ \{ \Delta_4^3 ( \vectau \cdot \vecpsi^\prime) \} \vectau(s_4) ] \cdot \Delta_4^3 \vecphi^\prime
\\
	&\quad
	-
	\Delta_4^3 \vectau \cdot [ \{ \Delta_4^3 (\vectau \cdot \vecpsi^\prime ) \} \vecphi^\prime(s_3)
	+ ( \vectau(s_4) \cdot \vecpsi^\prime(s_4) ) \Delta_4^3 \vecphi^\prime ]
\\
	&=
	\Delta_4^3 \vecphi^\prime \cdot \Delta_4^3 \vecpsi^\prime
\\
	&\quad
	-
	( \Delta_4^3 \vectau \cdot \Delta_4^3 \vecphi^\prime )
	( \vectau(s_3) \cdot \vecpsi^\prime(s_3) + \vectau(s_4) \cdot \vecpsi^\prime(s_4) )
\\
	&\quad
	- \{ \Delta_4^3 ( \vectau \cdot \vecpsi^\prime ) \}
	( \vectau(s_4) \cdot \Delta_4^3 \vecphi^\prime + \Delta_4^3 \vectau \cdot \vecphi^\prime(s_3) )
\\
	&=
	\Delta_4^3 \vecphi^\prime \cdot \Delta_4^3 \vecpsi^\prime
\\
	&\quad
	-
	( \Delta_4^3 \vectau \cdot \Delta_4^3 \vecphi^\prime )
	( \vectau(s_3) \cdot \vecpsi^\prime(s_3) + \vectau(s_4) \cdot \vecpsi^\prime(s_4) )
\\
	&\quad
	- \{ \Delta_4^3 ( \vectau \cdot \vecphi^\prime ) \} \{ \Delta_4^3 ( \vectau \cdot \vecpsi^\prime ) \}.
\end{align*}

\item Using Lemma \ref{variation lemma},
we have
\begin{align*}
	\delta \scrK(\vecf , \vecphi)[\vecpsi]
	&=
	\frac{ \delta ( \Delta \vecf \cdot \Delta \vecphi )[\vecpsi] }{ \| \Delta \vecf \|_{\R^n}^2 }
	+
	(\Delta \vecf \cdot \Delta \vecphi)
	\delta \left( \frac 1 { \| \Delta \vecf \|_{\R^n}^2 } \right) [\vecpsi]
\\
	&=
	\frac{ \Delta \vecpsi \cdot \Delta \vecphi }{ \| \Delta \vecf \|_{\R^n}^2 }
	-
	2
	\Delta \vecf \cdot \Delta \vecphi
	\frac{ \scrK(\vecf,\vecpsi) }{ \| \Delta \vecf \|_{\R^n}^2 }
\\
	&=
	\scrK(\vecphi,\vecpsi)
	-
	2 \scrK(\vecf,\vecphi) \scrK(\vecf,\vecpsi).
\end{align*}
\end{enumerate}
\end{proof}

\begin{lemma}\label{variation N}
$\delta\scrN(\vectau)[\vecphi]$
and
$\delta^2\scrN(\vectau)[\vecphi,\vecpsi]$
can be written as
\[
	\delta\scrN(\vectau)[\vecphi]
	=
	\sum_{i=1}^2 \scrR_i(\vecf)[\vecphi],
	\quad
	\delta^2\scrN(\vectau)[\vecphi,\vecpsi]
	=
	\sum_{i=1}^5 \scrS_i(\vecf)[\vecphi,\vecpsi],
\]
where
\begin{align*}
	\scrR_1 (\vecf)[\vecphi]
	&=
	-2\scrK(\vecf,\vecphi) \scrN(\vectau),
\\
	\scrR_2 (\vecf)[\vecphi]
	&=
	2\scrN(\vectau,\vecphi^\prime),
\\
	\scrS_1(\vecf)[\vecphi,\vecpsi]
	&=
	-2
	( \scrK(\vecphi,\vecpsi) -2 \scrK(\vecf,\vecphi) \scrK(\vecf,\vecpsi) )
	\scrN(\vectau),
\\
	\scrS_2(\vecf)[\vecphi,\vecpsi]
	&=
	-
	\scrK(\vecf,\vecphi) \delta \scrN(\vectau) [\vecpsi]
	-
	\scrK(\vecf,\vecpsi) \delta \scrN(\vectau) [\vecphi],
\\
	\scrS_3(\vecf)[\vecphi,\vecpsi]
	&=
	-2
	\scrK(\vecf,\vecpsi)
	\scrN(\vectau,\vecphi^\prime)
	-2
	\scrK(\vecf,\vecphi)
	\scrN(\vectau,\vecpsi^\prime),
\\
	\scrS_4(\vecf)[\vecphi,\vecpsi]
	&=
	2\scrN(\vecphi^\prime,\vecpsi^\prime),
\\
	\scrS_5(\vecf)[\vecphi,\vecpsi]
	&=
	-
	2 \scrN( (\vectau \cdot \vecphi^\prime) , (\vectau \cdot \vecpsi^\prime) ).
\end{align*}
\end{lemma}

\begin{proof}
By Lemmas \ref{variation lemma} and \ref{variation N sublemma},
we have
\begin{align*}
	\delta \scrN(\vectau)[\vecphi]
	&=
	-2\scrK(\vecf,\vecphi) \scrN(\vectau)
\\
	&\quad
	+
	\frac 1 { 2\| \Delta \vecf \|_{\R^n}^2 }
	\int_{s_2}^{s_1} \int_{s_2}^{s_1}
	\{
	\delta \| \Delta_4^3 \vectau \|_{\R^n}^2 [\vecphi]
\\
	&\qquad
	+
	\| \Delta_4^3 \vectau \|_{\R^n}^2
	(\vectau(s_3) \cdot \vecphi^\prime(s_3) + \vectau(s_4) \cdot \vecphi^\prime(s_4))
	\}
	ds_3ds_4
\\
	&=
	-2\scrK(\vecf,\vecphi) \scrN(\vectau)
	+
	\frac 1 { \| \Delta \vecf \|_{\R^n}^2 }
	\int_{s_2}^{s_1} \int_{s_2}^{s_1} \Delta_4^3 \vectau \cdot \Delta_4^3 \vecphi^\prime ds_3ds_4
\\
	&=
	-2\scrK(\vecf,\vecphi) \scrN(\vectau)
	+
	2\scrN(\vectau,\vecphi^\prime)
\\
	&=
	\sum_{i=1}^2 \scrR_i(\vecf)[\vecphi].
\end{align*}

Next we calculate
$\delta^2 \scrN(\vectau)[\vecphi,\vecpsi]$.
Firstly,
we have
\[
	\delta^2 \scrN(\vectau)[\vecphi,\vecpsi]
	=
	\frac 1 2 \delta ( \delta \scrN(\vectau)[\vecphi] ) [\vecpsi]
	+
	\frac 1 2 \delta ( \delta \scrN(\vectau)[\vecpsi] ) [\vecphi].
\]
By the symmetry with respect to
$\vecphi$
and
$\vecpsi$,
it is sufficient to calculate
$\delta ( \delta \scrN(\vectau)[\vecphi] ) [\vecpsi]$,
for which we first note that
\[
	\delta ( \delta \scrN(\vectau)[\vecphi] ) [\vecpsi]
	=
	\sum_{i=1}^2 \delta ( \scrR_i(\vecf)[\vecphi] ) [\vecpsi].
\]
Using Lemma \ref{variation N sublemma} again,
we can show that
\begin{align*}
	\delta (\scrR_1(\vecf)[\vecphi]) [\vecpsi]
	&=
	-2 \delta \scrK(\vecf,\vecphi) [\vecpsi] \scrN(\vectau)
	-2 \scrK(\vecf,\vecphi) \delta \scrN(\vectau) [\vecpsi]
\\
	&=
	-2
	\scrK(\vecphi,\vecpsi) \scrN(\vectau)
	+
	4
	\scrK(\vecf,\vecphi) \scrK(\vecf,\vecpsi)
	\scrN(\vectau)
\\
	&\quad
	-2
	\scrK(\vecf,\vecphi)
	\delta \scrN(\vectau)[\vecpsi],
\end{align*}
and
\begin{align*}
	\delta (\scrR_2(\vecf)[\vecphi]) [\vecpsi]
	&=
	2 \delta ( \scrN(\vectau,\vecphi^\prime) )[\vecpsi]
\\
	&=
	-4\scrK(\vecf,\vecpsi) \scrN(\vectau,\vecphi^\prime)
\\
	&\quad
	+
	\frac 1 { \| \Delta \vecf \|_{\R^n}^2 }
	\int_{s_2}^{s_1} \int_{s_2}^{s_1}
	[
	\Delta_4^3 \vecphi^\prime \cdot \Delta_4^3 \vecpsi^\prime
\\
	&\qquad
	-
	( \Delta_4^3 \vectau \cdot \Delta_4^3 \vecphi^\prime )
	( \vectau(s_3) \cdot \vecpsi^\prime(s_3) + \vectau(s_4) \cdot \vecpsi^\prime(s_4) )
\\
	&\qquad
	- \{ \Delta_4^3 ( \vectau \cdot \vecphi^\prime ) \} \{ \Delta_4^3 ( \vectau \cdot \vecpsi^\prime ) \}
\\
	&\qquad
	+
	( \Delta_4^3 \vectau \cdot \Delta_4^3 \vecphi^\prime )
	(\vectau(s_3) \cdot \vecpsi^\prime(s_3) + \vectau(s_4) \cdot \vecpsi^\prime(s_4))
	]
	ds_3ds_4
\\
	&=
	-4\scrK(\vecf,\vecpsi)
	\scrN(\vectau,\vecphi^\prime)
	+
	2\scrN(\vecphi^\prime,\vecpsi^\prime)
	-
	2 \scrN( (\vectau \cdot \vecphi^\prime) , (\vectau \cdot \vecpsi^\prime) ),
\end{align*}
from which the claim follows.
\end{proof}

Next,
we give expressions of
$\delta \scrM_\alpha(\vecf) [\vecphi]$
and
$\delta^2 \scrM_\alpha(\vecf) [\vecphi,\vecpsi]$
in terms of
$\scrN$,
$\delta \scrN$,
and
$\delta^2 \scrN$.

\begin{proposition}\label{variation Malpha}
$\delta \scrM_\alpha(\vecf) [\vecphi]$
and
$\delta^2 \scrM_\alpha(\vecf) [\vecphi,\vecpsi]$
can be written as
\[
	\delta \scrM_\alpha(\vecf) [\vecphi]
	=
	\sum_{i=1}^2 \scrP_i(\vecf) [\vecphi],
	\quad
	\delta^2 \scrM_\alpha(\vecf) [\vecphi,\vecpsi]
	=
	\sum_{i=1}^6 \scrQ_i(\vecf)[\vecphi,\vecpsi],
\]
where
\begin{align*}
	\scrP_1(\vecf)[\vecphi]
	&=
	\vphi_\alpha^\prime(\scrN(\vectau))
	\frac{ \delta \scrN(\vectau)[\vecphi] }{ \| \Delta \vecf \|_{\R^n}^\alpha },
\\
	\scrP_2(\vecf)[\vecphi]
	&=
	-
	\frac \alpha 2 
	\scrM_\alpha(\vecf)
	\frac{ \delta \| \Delta \vecf \|_{\R^n}^2 [\vecphi] }{ \| \Delta \vecf \|_{\R^n}^2 },
\\
	\scrQ_1(\vecf)[\vecphi,\vecpsi]
	&=
	\vphi_\alpha^\prime(\scrN(\vectau))
	\frac{ \delta^2 \scrN(\vectau)[\vecphi,\vecpsi] }{ \| \Delta \vecf \|_{\R^n}^\alpha },
\\
	\scrQ_2(\vecf)[\vecphi,\vecpsi]
	&=
	-
	\frac \alpha 2
	\vphi_\alpha^\prime(\scrN(\vectau))
	\frac{ \delta \scrN(\vectau)[\vecphi] }{ \| \Delta \vecf \|_{\R^n}^\alpha }
	\frac{ \delta \| \Delta \vecf \|_{\R^n}^2 [\vecpsi] }{ \| \Delta \vecf \|_{\R^n}^2 },
\\
	\scrQ_3(\vecf)[\vecphi,\vecpsi]
	&=
	\vphi_\alpha^{\prime\prime}(\scrN(\vectau))
	\frac{ \delta \scrN(\vectau)[\vecphi] }{ \| \Delta \vecf \|_{\R^n}^{\frac \alpha 2} }
	\frac{ \delta \scrN(\vectau)[\vecpsi] }{ \| \Delta \vecf \|_{\R^n}^{\frac \alpha 2} },
\\
	\scrQ_4(\vecf)[\vecphi,\vecpsi]
	&=
	-
	\frac \alpha 2
	\delta \scrM_\alpha(\vecf)[\vecpsi]
	\frac{ \delta \| \Delta \vecf \|_{\R^n}^2 [\vecphi] }{ \| \Delta \vecf \|_{\R^n}^2 },
\\
	\scrQ_5(\vecf)[\vecphi,\vecpsi]
	&=
	-
	\frac \alpha 2
	\scrM_\alpha(\vecf)
	\frac{ \delta^2 \| \Delta \vecf \|_{\R^n}^2 [\vecphi,\vecpsi] }{ \| \Delta \vecf \|_{\R^n}^2 },
\\
	\scrQ_6(\vecf)[\vecphi,\vecpsi]
	&=
	\frac \alpha 2
	\scrM_\alpha(\vecf)
	\frac{ \delta \| \Delta \vecf \|_{\R^n}^2 [\vecphi] }{ \| \Delta \vecf \|_{\R^n}^2 }
	\frac{ \delta \| \Delta \vecf \|_{\R^n}^2 [\vecpsi] }{ \| \Delta \vecf \|_{\R^n}^2 }.
\end{align*}
\end{proposition}

\begin{proof}
The assertion for
$\delta \scrM_\alpha(\vecf)$
follows immediately from \eqref{def Malpha} and \eqref{delta Malpha}.

Regarding
$\delta^2 \scrM_\alpha(\vecf)$,
we have
\[
	\delta^2\scrM_\alpha(\vecf)[\vecphi,\vecpsi]
	=
	\delta ( \delta\scrM_\alpha(\vecf)[\vecphi] )[\vecpsi]
	=
	\sum_{i=1}^2
	\delta (\scrP_i(\vecf)[\vecphi]) [\vecpsi],
\]
and
\begin{align*}
	\delta (\scrP_1(\vecf)[\vecphi]) [\vecpsi]
	&=
	\vphi_\alpha^{\prime\prime}(\scrN(\vectau)) \delta\scrN(\vectau)[\vecpsi]
	\frac{ \delta \scrN(\vectau)[\vecphi] }{ \| \Delta \vecf \|_{\R^n}^\alpha }
	+
	\vphi_\alpha^\prime(\scrN(\vectau))
	\frac{ \delta^2 \scrN(\vectau)[\vecphi,\vecpsi] }{ \| \Delta \vecf \|_{\R^n}^\alpha }
\\
	&\quad
	-
	\frac \alpha 2
	\vphi_\alpha^\prime(\scrN(\vectau))
	\frac{ \delta \scrN(\vectau)[\vecphi] \delta\| \Delta \vecf \|_{\R^n}^2[\vecpsi] }
	{ \| \Delta \vecf \|_{\R^n}^{\alpha+2} }
\\
	&=
	\scrQ_3(\vecf)[\vecphi,\vecpsi] + \scrQ_1(\vecf)[\vecphi,\vecpsi] + \scrQ_2(\vecf)[\vecphi,\vecpsi],
\\
	\delta (\scrP_2(\vecf)[\vecphi]) [\vecpsi]
	&=
	-
	\frac \alpha 2
	\delta\scrM_\alpha(\vecf)[\vecpsi]
	\frac{ \delta \| \Delta \vecf \|_{\R^n}^2 [\vecphi] }{ \| \Delta \vecf \|_{\R^n}^2 }
	-
	\frac \alpha 2
	\scrM_\alpha(\vecf)
	\frac{ \delta^2 \| \Delta \vecf \|_{\R^n}^2 [\vecphi,\vecpsi] }{ \| \Delta \vecf \|_{\R^n}^2 }
\\
	&\quad
	+
	\frac \alpha 2
	\scrM_\alpha(\vecf)
	\frac{ \delta \| \Delta \vecf \|_{\R^n}^2 [\vecphi] \delta \| \Delta \vecf \|_{\R^n}^2 [\vecpsi] }
	{ \| \Delta \vecf \|_{\R^n}^4 }
\\
	&=
	\scrQ_4(\vecf)[\vecphi,\vecpsi] + \scrQ_5(\vecf)[\vecphi,\vecpsi] + \scrQ_6(\vecf)[\vecphi,\vecpsi].
\end{align*}
\end{proof}

Plugging Lemma \ref{variation N} and Proposition \ref{variation Malpha} into Theorem \ref{variation},
we obtain the first and second variational formulae of the
$(\alpha,p)$-O'Hara energies.
Since the expressions are rather lengthy,
we omit their explicit formulae here.

%
%

\section{Estimates of the first and second variational formulae}\label{estimates}

\subsection{Strategy}

Let
$\vecf$,
$\vecphi$,
$\vecpsi \in W^{1+\sigma , 2p}(\R/\calL\Z , \R^n) \cap W^{1,\infty}(\R/\calL\Z , \R^n)$,
and assume that
$\vecf$
is bi-Lipschitz.
By \cite{B12},
we already know the estimate
\begin{equation}\label{Malpha Lp}
	\| \scrM_\alpha(\vecf) \|_{L^p} \leq C.
\end{equation}
Combining Theorem \ref{variation} and H\"{o}lder's inequality,
if
$p>1$,
we have
\begin{align*}
	\| \scrG_1(\vecf)[\vecphi] \|_{L^1}
	&\leq
	p \| \scrM_\alpha(\vecf) \|_{L^p}^{p-1} \| \delta \scrM_\alpha(\vecf)[\vecphi] \|_{L^p},
\\
	\| \scrG_2(\vecf)[\vecphi] \|_{L^1}
	&\leq
	2 \| \scrM_\alpha(\vecf) \|_{L^p}^p \| \vecphi^\prime \|_{L^\infty},
\\
	\| \scrH_1(\vecf)[\vecphi,\vecpsi] \|_{L^1}
	&\leq
	p \| \scrM_\alpha(\vecf) \|_{L^p}^{p-1} \| \delta^2 \scrM_\alpha(\vecf) [\vecphi,\vecpsi] \|_{L^p},
\\
	\| \scrH_2(\vecf)[\vecphi,\vecpsi] \|_{L^1}
	&\leq
	p(p-1) \| \scrM_\alpha(\vecf) \|_{L^p}^{p-2}
	\| \delta \scrM_\alpha(\vecf) [\vecphi] \|_{L^p} \| \delta \scrM_\alpha(\vecf) [\vecpsi] \|_{L^p},
\\
	\| \scrH_3(\vecf)[\vecphi,\vecpsi] \|_{L^1}
	&\leq
	2 \| \scrG_1(\vecf)[\vecphi] \|_{L^1} \| \vecpsi^\prime \|_{L^\infty},
\\
	\| \scrH_4(\vecf)[\vecphi,\vecpsi] \|_{L^1}
	&\leq
	2 \| \scrG_1(\vecf)[\vecpsi] \|_{L^1} \| \vecphi^\prime \|_{L^\infty},
\\
	\| \scrH_5(\vecf)[\vecphi,\vecpsi] \|_{L^1}
	&\leq
	6 \| \scrM_\alpha(\vecf) \|_{L^p}^p
	\| \vecphi^\prime \|_{L^\infty} \| \vecpsi^\prime \|_{L^\infty},
\\
	\| \scrH_6(\vecf)[\vecphi,\vecpsi] \|_{L^1}
	&\leq
	4 \| \scrM_\alpha(\vecf) \|_{L^p}^p
	\| \vecphi^\prime \|_{L^\infty} \| \vecpsi^\prime \|_{L^\infty}.
\end{align*}
Hence,
if there exists
$C=C(\vecf)>0$
such that
\begin{equation}\label{what to prove Malpha}
\begin{split}
	\| \delta \scrM_\alpha(\vecf)[\vecphi] \|_{L^p}
	&\leq C \| \vecphi^\prime \|_{W^{\sigma,2p} \cap L^\infty},
\\
	\| \delta^2 \scrM_\alpha(\vecf)[\vecphi,\vecpsi] \|_{L^p}
	&\leq
	C \| \vecphi^\prime \|_{W^{\sigma,2p} \cap L^\infty}
	\| \vecpsi^\prime \|_{W^{\sigma,2p} \cap L^\infty},
\end{split}
\end{equation}
then the desired
$L^1$-estimates for
$\scrG_{(\alpha,p)}(\vecf)$
and
$\scrH_{(\alpha,p)}(\vecf)$
follow.

Next we observe that since
$\| \Delta \vecf \|_{\R^n} \leq \scrD(\vecf)$
holds,
and using the bi-Lipschitz estimate,
we deduce that there exists
$\tilde{C}=\tilde{C}(\vecf)>0$
such that
$0 \leq \scrN(\vectau) \leq \tilde{C}$.
Hence,
we have
\[
	| \vphi_\alpha^{(j)} (\scrN(\vectau)) | \leq
	\max_{t \in [0,\tilde{C}]} | \vphi_\alpha^{(j)} (t) | < \infty
\]
for
$j=0,1,2$.
By Proposition \ref{variation Malpha},
H\"{o}lder's inequality,
and \eqref{Malpha Lp},
we can show \eqref{what to prove Malpha} if there exists
$C=C(\vecf)>0$
such that
\begin{align*}
	\left\|
	\frac{ \delta \scrN(\vectau)[\vecphi] }{ \| \Delta \vecf \|_{\R^n}^\alpha }
	\right\|_{L^p}
	&\leq
	C \| \vecphi^\prime \|_{W^{\sigma,2p} \cap L^\infty},
\\
	\left\|
	\frac{ \delta^2 \scrN(\vectau)[\vecphi,\vecpsi] }{ \| \Delta \vecf \|_{\R^n}^\alpha }
	\right\|_{L^p}
	&\leq
	C \| \vecphi^\prime \|_{W^{\sigma,2p} \cap L^\infty}
	\| \vecpsi^\prime \|_{W^{\sigma,2p} \cap L^\infty},
\\
	\left\|
	\frac{ \delta\| \Delta \vecf \|_{\R^d}^2[\vecphi] }{ \| \Delta \vecf \|_{\R^n}^2 }
	\right\|_{L^\infty}
	&\leq
	C \| \vecphi^\prime \|_{L^\infty},
\\
	\left\|
	\frac{ \delta^2\| \Delta \vecf \|_{\R^d}^2[\vecphi,\vecpsi] }{ \| \Delta \vecf \|_{\R^n}^2 }
	\right\|_{L^\infty}
	&\leq
	C \| \vecphi^\prime \|_{L^\infty} \| \vecpsi^\prime \|_{L^\infty}.
\end{align*}
Similarly,
to obtain the desired
$L^\infty$-estimates and continuity,
it suffices to consider the corresponding properties of
\begin{equation}\label{what to prove}
	\frac{ \delta \scrN(\vectau)[\vecphi] }{ \| \Delta \vecf \|_{\R^n}^\alpha },
	\quad
	\frac{ \delta^2 \scrN(\vectau)[\vecphi,\vecpsi] }{ \| \Delta \vecf \|_{\R^n}^\alpha },
	\quad
	\frac{ \delta\| \Delta \vecf \|_{\R^d}^2[\vecphi] }{ \| \Delta \vecf \|_{\R^n}^2 },
	\quad
	\frac{ \delta^2\| \Delta \vecf \|_{\R^d}^2[\vecphi,\vecpsi] }{ \| \Delta \vecf \|_{\R^n}^2 }.
\end{equation}

\subsection{Estimates and continuity of the quantities in \eqref{what to prove}}

As we can see from Lemma \ref{variation N},
$\delta \scrN(\vectau)[\vecphi]$
and
$\delta \scrN(\vectau)[\vecphi,\vecpsi]$
may be expressed in terms of
$\scrN(\vecu,\vecv)$
and
$\scrK(\vecu,\vecv)$.
First,
we discuss
$\scrN(\vecu,\vecv)$.

\begin{lemma}\label{estimates N}
Assume that
$\vecf$
is bi-Lipschitz.
Then,
the following properties hold.
\begin{enumerate}\renewcommand{\labelenumi}{\upshape\textrm{\theenumi.}}
\item
Let
$\alpha \in (0,\infty)$
and
$p \in [1,\infty)$
satisfy
$2 \leq \alpha p < 2p+1$.
If
$\vecu$,
$\vecv \in W^{\sigma,2p}(\R/\calL\Z , \R^d)$,
then there exists
$C=C(\vecf)>0$
such that
\[
	\left\|
	\frac{ \scrN(\vecu,\vecv) }{ \| \Delta \vecf \|_{\R^n}^\alpha }
	\right\|_{L^p}
	\leq
	C [\vecu]_{W^{\sigma,2p}} [\vecv]_{W^{\sigma,2p}}.
\]
\item
Let
$0 < \beta \leq 1$.
If
$\vecu$,
$\vecv \in C^{0,\beta}(\R/\calL\Z , \R^d)$,
then there exists
$C=C(\vecf)>0$
such that
\[
	\left\|
	\scrD(\vecf)^{\alpha-2\beta}
	\frac{ \scrN(\vecu,\vecv) }{ \| \Delta \vecf \|_{\R^n}^\alpha }
	\right\|_{L^\infty}
	\leq
	C [\vecu]_{C^{0,\beta}} [\vecv]_{C^{0,\beta}}.
\]
\item
Let
$0 < \beta \leq 1$.
If
$\vecu$,
$\vecv \in X^\beta(\R/\calL\Z , \R^d)$,
then for
$s \in \R/\calL\Z$,
\[
	\lim_{(s_1,s_2) \to (s,s)}
	|\Delta s|^{\alpha-2\beta}
	\frac{ \scrN(\vecu,\vecv) }{ \| \Delta \vecf \|_{\R^n}^\alpha }
	=
	\left\{
	\begin{array}{ll}
		0 & (0<\beta<1),
	\vspace{5pt}\\
		\displaystyle \frac 1 {12} \vecu^\prime(s) \cdot \vecv^\prime(s) & (\beta=1)
	\end{array}
	\right.
\]
holds.
\end{enumerate}
\end{lemma}

\begin{proof}
\begin{enumerate}
\item
By H\"{o}lder's inequality and the bi-Lipschitz continuity of
$\vecf$,
we have
\begin{align*}
	\left\|
	\frac{ \scrN(\vecu,\vecv) }{ \| \Delta \vecf \|_{\R^n}^\alpha }
	\right\|_{L^p}^p
	&=
	\frac 1 {2^p}
	\int_{\R/\calL\Z} \int_{s_2-\frac \calL 2}^{s_2+\frac \calL 2}
	\frac 1 { \| \Delta \vecf \|_{\R^n}^{\alpha p+2p} }
	\left|
	\int_{s_2}^{s_1} \int_{s_2}^{s_1}
	\Delta_4^3 \vecu \cdot \Delta_4^3 \vecv
	ds_3ds_4
	\right|^p
	ds_1ds_2
\\
	&\leq
	C
	\int_{\R/\calL\Z} \int_{s_2-\frac \calL 2}^{s_2+\frac \calL 2}
	\frac 1 { |\Delta s|^{\alpha p+2} }
	\int_{s_2}^{s_1} \int_{s_2}^{s_1}
	\| \Delta_4^3 \vecu \|_{\R^d}^p \| \Delta_4^3 \vecv \|_{\R^d}^p
	ds_3ds_4ds_1ds_2
\\
	&=
	(\dagger).
\end{align*}
We change variables
\[
	t_1=s_1-s_2,
	\quad
	t_2=s_2,
	\quad
	s_3=t_2+t_1t_3,
	\quad
	s_4=t_2+t_1t_4
\]
in
$(\dagger)$.
Then,
we obtain
\begin{multline*}
	(\dagger)
	=
	C
	\int_{\R/\calL\Z} \int_{-\frac \calL 2}^{\frac \calL 2}
	\frac 1 { |t_1|^{\alpha p} }
\\
	\times
	\int_0^1 \int_0^1
	\| \vecu(t_2+t_1t_3) - \vecu(t_2+t_1t_4) \|_{\R^d}^p
	\| \vecv(t_2+t_1t_3) - \vecv(t_2+t_1t_4) \|_{\R^d}^p
	dt_3dt_4dt_1dt_2
\\
	=
	(\ddagger).
\end{multline*}
We use Fubini's theorem and change variables
\[
	w_1 = (t_3-t_4) t_1,
	\quad
	w_2 = t_2+t_1t_4
\]
in
$(\ddagger)$.
Then,
we obtain
\begin{align*}
	(\ddagger)
	&=
	C
	\int_0^1 \int_0^1
	\int_{-\frac \calL 2}^{\frac \calL 2} \int_{\R/\calL\Z}
	\frac 1 { |t_1|^{\alpha p} }
\\
	&\qquad
	\times
	\| \vecu(t_2+t_1t_3) - \vecu(t_2+t_1t_4) \|_{\R^d}^p
	\| \vecv(t_2+t_1t_3) - \vecv(t_2+t_1t_4) \|_{\R^d}^p
	dt_2dt_1dt_3dt_4
\\
	&=
	C
	\int_0^1 \int_0^1
	\int_{-\frac \calL 2 |t_3-t_4|}^{\frac \calL 2 |t_3-t_4|} \int_{\R/\calL\Z}
	\frac 1 { |w_1|^{\alpha p} }
	|t_3-t_4|^{\alpha p-1}
\\
	&\qquad
	\times
	\| \vecu(w_1+w_2) - \vecu(w_2) \|_{\R^d}^p
	\| \vecv(w_1+w_2) - \vecv(w_2) \|_{\R^d}^p
	dw_1dw_2dt_3dt_4
\\
	&\leq
	C
	\int_{\R/\calL\Z} \int_{-\frac \calL 2}^{\frac \calL 2}
	\frac{ \| \vecu(w_1+w_2) - \vecu(w_2) \|_{\R^d}^p }{ |w_1|^{\frac{\alpha p} 2} }
	\frac{ \| \vecv(w_1+w_2) - \vecv(w_2) \|_{\R^d}^p }{ |w_1|^{\frac{\alpha p} 2} }
	dw_1dw_2
\\
	&\leq
	C [\vecu]_{W^{\sigma,2p}}^p [\vecv]_{W^{\sigma,2p}}^p,
\end{align*}
by H\"{o}lder's inequality,
and the claim holds.
\item
Without loss of generality,
we assume that
$s_1$,
$s_2 \in \R/\calL\Z$
satisfy
$\displaystyle 0 < |\Delta s| < \calL/2$.
Then,
we have
\begin{align*}
	\left|
	|\Delta s|^{\alpha-2\beta}
	\frac{ \scrN(\vecu,\vecv) }{ \| \Delta \vecf \|_{\R^n}^\alpha }
	\right|
	&\leq
	C
	|\Delta s|^{-2\beta-2}
	\int_{s_2}^{s_1} \int_{s_2}^{s_1}
	\| \Delta_4^3 \vecu \|_{\R^d} \| \Delta_4^3 \vecv \|_{\R^d}
	ds_3ds_4
\\
	&\leq
	C
	|\Delta s|^{-2}
	\int_{s_2}^{s_1} \int_{s_2}^{s_1}
	\frac{ \| \Delta_4^3 \vecu \|_{\R^d} }{ |\Delta_4^3 s|^\beta }
	\frac{ \| \Delta_4^3 \vecv \|_{\R^d} }{ |\Delta_4^3 s|^\beta }
	ds_3ds_4
\\
	&\leq
	C [\vecu^\prime]_{C^{0,\beta}} [\vecv^\prime]_{C^{0,\beta}}.
\end{align*}
\item
First,
we consider the case where
$0<\beta<1$,
i.e.\ $\vecu$,
$\vecv \in h^{0,\beta}(\R/\calL\Z , \R^d)$.
As is well known,
the little H\"{o}lder space
$h^{0,\beta}$
is characterized as
\[
	h^{0,\beta}(\R/\calL\Z , \R^d)
	=
	\left\{
	\vecu \in C^{0,\beta}(\R/\calL\Z , \R^d)
	\,\left|\,
	\lim_{R \to +0} [\vecu]_{C^{0,\beta},R} = 0
	\right.
	\right\},
\]
where
\[
	[\vecu]_{C^{0,\beta},R}
	=
	\sup_{\substack{s_1,s_2 \in \R/\calL\Z\\0<|\Delta s|<R}}
	\frac{ \| \Delta \vecu \|_{\R^d} }{ |\Delta s|^\beta }.
\]
For
$R>0$,
let
$s_1$,
$s_2 \in \R/\calL\Z$
satisfy
$0<|\Delta s|<R$.
Then,
we have
\begin{equation}\label{Req}
	\left|
	|\Delta s|^{\alpha-2\beta}
	\frac{ \scrN(\vecu,\vecv) }{ \| \Delta \vecf \|_{\R^n}^\alpha }
	\right|
	\leq
	C [\vecu]_{C^{0,\beta},R} [\vecv]_{C^{0,\beta},R}.
\end{equation}
Taking
$\limsup$
on left-hand side in \eqref{Req},
we have
\[
	\limsup_{(s_1,s_2) \to (s,s)}
	\left|
	|\Delta s|^{\alpha-2\beta}
	\frac{ \scrN(\vecu,\vecv) }{ \| \Delta \vecf \|_{\R^n}^\alpha }
	\right|
	\leq
	C [\vecu]_{C^{0,\beta},R} [\vecv]_{C^{0,\beta},R}.
\]
Taking
$R \to +0$
on right-hand side,
we obtain
\[
	\lim_{(s_1,s_2) \to (s,s)}
	|\Delta s|^{\alpha-2\beta}
	\frac{ \scrN(\vecu,\vecv) }{ \| \Delta \vecf \|_{\R^n}^\alpha }
	= 0.
\]
Next,
we assume
$\vecu$,
$\vecv \in C^1$.
If
$s_5$,
$s_6 \in \R/\calL\Z$
are such that
$\| (s_5,s_6) - (s,s) \|_{\R^2}$
is sufficiently small,
we can take
$\epsi>0$
arbitrarily such that
\[
	\| \vecu^\prime(s_5) - \vecu^\prime(s) \|_{\R^d}
	\leq
	\frac \epsi { 2( \| \vecv^\prime \|_{L^\infty} + 1) },
	\quad
	\| \vecv^\prime(s_6) - \vecv^\prime(s) \|_{\R^d}
	\leq
	\frac \epsi { 2( \| \vecu^\prime \|_{L^\infty} + 1) }.
\]
Then,
using the fact that
\[
	\int_{s_2}^{s_1} \int_{s_2}^{s_1} \int_{s_4}^{s_3} \int_{s_4}^{s_3} ds_5ds_6ds_3ds_4 = \frac 1 6 |\Delta s|^2,
\]
we have
\begin{align*}
	&
	\left|
	\frac 1 { |\Delta s|^4 }
	\int_{s_2}^{s_1} \int_{s_2}^{s_1} \int_{s_4}^{s_3} \int_{s_4}^{s_3}
	\vecu^\prime(s_5) \cdot \vecv^\prime(s_6)
	ds_5ds_6ds_3ds_4
	-
	\frac 1 6 \vecu^\prime(s) \cdot \vecv^\prime(s)
	\right|
\\
	&\leq
	\frac 1 { |\Delta s|^4 }
	\int_{s_2}^{s_1} \int_{s_2}^{s_1} \int_{s_4}^{s_3} \int_{s_4}^{s_3}
	| \vecu^\prime(s_5) \cdot \vecv^\prime(s_6) - \vecu^\prime(s) \cdot \vecv^\prime(s) |
	ds_5ds_6ds_3ds_4
\\
	&\leq
	\frac 1 { |\Delta s|^4 }
	\int_{s_2}^{s_1} \int_{s_2}^{s_1} \int_{s_4}^{s_3} \int_{s_4}^{s_3}
	( \| \vecu^\prime(s_5) \|_{\R^d} \| \vecv^\prime(s_6) - \vecv^\prime(s) \|_{\R^d}
\\
	&\qquad
	+
	\| \vecv^\prime(s) \|_{\R^d} \| \vecu^\prime(s_5) - \vecu^\prime(s) \|_{\R^d} )
	ds_5ds_6ds_3ds_4
\\
	&\leq
	\epsi.
\end{align*}
Hence,
we have
\begin{multline*}
	\lim_{(s_1,s_2) \to (s,s)}
	\frac 1 { |\Delta s|^4 }
	\int_{s_2}^{s_1} \int_{s_2}^{s_1} \int_{s_4}^{s_3} \int_{s_4}^{s_3}
	\vecu^\prime(s_5) \cdot \vecv^\prime(s_6)
	ds_5ds_6ds_3ds_4
\\
	=
	\frac 1 6 \vecu^\prime(s) \cdot \vecv^\prime(s).
\end{multline*}
Using this and
\[
	\lim_{(s_1,s_2) \to (s,s)}
	\frac{ |\Delta s| }{ \| \Delta \vecf \|_{\R^n} }
	=
	\frac 1 { \| \vectau(s) \|_{\R^n} }
	= 1,
\]
we obtain
\begin{align*}
	&
	|\Delta s|^{\alpha-2} \frac{ \scrN(\vecu,\vecv) }{ \| \Delta \vecf \|_{\R^n}^\alpha }
\\
	&=
	\frac 1 2
	\frac{ |\Delta s|^{\alpha+2} }{ \| \Delta \vecf \|_{\R^n}^{\alpha+2} }
	\frac 1 { |\Delta s|^4 }
	\int_{s_2}^{s_1} \int_{s_2}^{s_1} \int_{s_4}^{s_3} \int_{s_4}^{s_3}
	\vecu^\prime(s_5) \cdot \vecv^\prime(s_6)
	ds_5ds_6ds_3ds_4
\\
	&\to
	\frac 1 {12} \vecu^\prime(s) \cdot \vecv^\prime(s)
	\quad
	\text{as}\ (s_1,s_2) \to (s,s).
\end{align*}
\end{enumerate}
\end{proof}

Since the function
$\scrK(\vecu,\vecv)$
fulfills
\[
	| \scrK(\vecu,\vecv) | \leq C \| \vecu^\prime \|_{L^\infty} \| \vecv^\prime \|_{L^\infty},
	\quad
	\lim_{(s_1,s_2) \to (s,s)} \scrK(\vecu,\vecv)
	=
	\vecu^\prime(s) \cdot \vecv^\prime(s),
\]
we obtain estimates and continuity of
$\displaystyle \frac{ \delta \| \Delta \vecf \|_{\R^n}^2 [\vecphi] }{ \| \Delta \vecf \|_{\R^n}^2 }$
and
$\displaystyle \frac{ \delta^2 \| \Delta \vecf \|_{\R^n}^2 [\vecphi,\vecpsi] }{ \| \Delta \vecf \|_{\R^n}^2 }$
from Lemma \ref{variation lemma} as follows.

\begin{lemma}\label{estimates deltaDelta}
\begin{enumerate}\renewcommand{\labelenumi}{\upshape\textrm{\theenumi.}}
\item
Assume that
$\vecf$,
$\vecphi$,
$\vecpsi \in W^{1,\infty}(\R/\calL\Z,\R^n)$
and
$\vecf$
is bi-Lipschitz.
Then,
there exists
$C=C(\vecf)>0$
such that
\[
	\left\| \frac{ \delta \| \Delta \vecf \|_{\R^n}^2[\vecphi] }{ \| \Delta \vecf \|_{\R^n}^2 } \right\|_{L^\infty}
	\leq
	C \| \vecphi^\prime \|_{L^\infty},
	\quad
	\left\|
	\frac{ \delta^2 \| \Delta \vecf \|_{\R^n}^2[\vecphi,\vecpsi] }{ \| \Delta \vecf \|_{\R^n}^2 }
	\right\|_{L^\infty}
	\leq
	C \| \vecphi^\prime \|_{L^\infty} \| \vecpsi^\prime \|_{L^\infty}.
\]
\item
If
$\vecf$,
$\vecphi$,
$\vecpsi \in C^1(\R/\calL\Z,\R^n)$,
then it follows that
\begin{align*}
	\lim_{(s_1,s_2) \to (s,s)}
	\frac{ \delta \| \Delta \vecf \|_{\R^n}^2[\vecphi] }{ \| \Delta \vecf \|_{\R^n}^2 }
	&=
	\vectau(s) \cdot \vecphi^\prime(s),
\\
	\lim_{(s_1,s_2) \to (s,s)}
	\frac{ \delta^2 \| \Delta \vecf \|_{\R^n}^2[\vecphi,\vecpsi] }{ \| \Delta \vecf \|_{\R^n}^2 }
	&=
	\vecphi^\prime(s) \cdot \vecpsi^\prime(s)
\end{align*}
for
$s \in \R/\calL\Z$.
\end{enumerate}
\end{lemma}

To deal with
$\scrS_5(\vecf)$,
we need estimates for
$\vectau \cdot \vecphi^\prime$.
The proof is easy,
therefore we omit the details.

\begin{lemma}\label{estimates inner product}
Assume that
$\vecf$
is bi-Lipschitz.
Then,
the following properties hold.
\begin{enumerate}\renewcommand{\labelenumi}{\upshape\textrm{\theenumi.}}
\item
Let
$\alpha \in (0,\infty)$
and
$p \in [1,\infty)$
satisfy
$2 \leq \alpha p < 2p+1$.
If
$\vecf$,
$\vecphi \in W^{1+\sigma,2p}(\R/\calL\Z , \R^n) \cap W^{1,\infty}(\R/\calL\Z , \R^d)$,
then there exists
$C=C(\vecf)>0$
such that
\[
	[(\vectau \cdot \vecphi^\prime)]_{W^{\sigma,2p}}
	\leq
	C
	( \| \vectau \|_{L^\infty} [\vecphi^\prime]_{W^{\sigma,2p}}
	+ [\vectau]_{W^{\sigma,2p}} \| \vecphi^\prime \|_{L^\infty} ).
\]
\item
Let
$0 < \beta \leq 1$.
If
$\vecf$,
$\vecphi \in C^{1,\beta}(\R/\calL\Z , \R^n)$,
then it holds that
\[
	[(\vectau \cdot \vecphi^\prime)]_{C^{0,\beta}}
	\leq
	\| \vectau \|_{L^\infty} [\vecphi^\prime]_{C^{0,\beta}} + [\vectau]_{C^{0,\beta}} \| \vecphi^\prime \|_{L^\infty}.
\]
\item
Let
$0 < \beta \leq 1$.
If
$\vectau$,
$\vecphi^\prime \in X^\beta(\R/\calL\Z , \R^n)$,
then
$(\vectau \cdot \vecphi^\prime) \in X^\beta(\R/\calL\Z , \R)$.
Moreover,
if
$\vecf$,
$\vecphi \in C^2(\R/\calL\Z , \R^n)$,
then it holds that
\[
	\frac d {ds} (\vectau(s) \cdot \vecphi^\prime(s))
	=
	\veckappa(s) \cdot \vecphi^\prime(s) + \vectau(s) \cdot \vecphi^{\prime\prime}(s)
\]
for
$s \in \R/\calL\Z$.
\end{enumerate}
\end{lemma}

\subsection{Proof of Theorem \ref{mainthm}}

In this subsection we complete the proof of Theorem \ref{mainthm} combining the facts in previous subsection.
First,
we show the
$L^1$-estimates of the first and second variational formulae for
$\scrM_{(\alpha,p)}(\vecf)$.
Using the expression of the first and second variational formulae for
$\scrN(\vectau)$
in Lemma \ref{variation N},
it follows that
\begin{align*}
	\left\| \frac{ \scrR_1(\vecf)[\vecphi] }{ \| \Delta \vecf \|_{\R^n}^\alpha } \right\|_{L^p}
	&\leq
	C \| \vecphi^\prime \|_{L^\infty},
\\
	\left\| \frac{ \scrR_2(\vecf)[\vecphi] }{ \| \Delta \vecf \|_{\R^n}^\alpha } \right\|_{L^p}
	&\leq
	C [ \vecphi^\prime ]_{W^{\sigma,2p}},
\\
	\left\| \frac{ \scrS_1(\vecf)[\vecphi,\vecpsi] }
	{ \| \Delta \vecf \|_{\R^n}^\alpha } \right\|_{L^p}
	&\leq
	C \| \vecphi^\prime \|_{L^\infty} \| \vecpsi^\prime \|_{L^\infty},
\\
	\left\| \frac{ \scrS_2(\vecf)[\vecphi,\vecpsi] }
	{ \| \Delta \vecf \|_{\R^n}^\alpha } \right\|_{L^p}
	&\leq
	C
	\{
	\| \vecphi^\prime \|_{L^\infty}
	( \| \vecpsi^\prime \|_{L^\infty} + [\vecpsi^\prime]_{W^{\sigma,2p}} )
	+
	\| \vecpsi^\prime \|_{L^\infty}
	( \| \vecphi^\prime \|_{L^\infty} + [\vecphi^\prime]_{W^{\sigma,2p}} )
	\},
\\
	\left\| \frac{ \scrS_3(\vecf)[\vecphi,\vecpsi] }
	{ \| \Delta \vecf \|_{\R^n}^\alpha } \right\|_{L^p}
	&\leq
	C ( \| \vecpsi^\prime \|_{L^\infty} [ \vecphi^\prime ]_{W^{\sigma,2p}}
	+
	\| \vecphi^\prime \|_{L^\infty} [ \vecpsi^\prime ]_{W^{\sigma,2p}} ),
\\
	\left\| \frac{ \scrS_4(\vecf)[\vecphi,\vecpsi] }
	{ \| \Delta \vecf \|_{\R^n}^\alpha } \right\|_{L^p}
	&\leq
	C [\vecphi^\prime]_{W^{\sigma,2p}} [\vecpsi^\prime]_{W^{\sigma,2p}},
\\
	\left\| \frac{ \scrS_5(\vecf)[\vecphi,\vecpsi] }
	{ \| \Delta \vecf \|_{\R^n}^\alpha } \right\|_{L^p}
	&\leq
	C ( [\vecphi^\prime]_{W^{\sigma,2p}} + \| \vecphi^\prime \|_{L^\infty} )
	( [\vecpsi^\prime]_{W^{\sigma,2p}} + \| \vecpsi^\prime \|_{L^\infty} )
\end{align*}
by Lemmas \ref{estimates N} and \ref{estimates inner product}.
Hence,
from these estimates and Lemma \ref{estimates deltaDelta},
we obtain
\begin{align*}
	\left\|
	\frac{ \delta \scrN(\vectau)[\vecphi] }{ \| \Delta \vecf \|_{\R^n}^\alpha }
	\right\|_{L^p}
	&\leq
	C \| \vecphi^\prime \|_{W^{\sigma,2p} \cap L^\infty}
\\
	\left\|
	\frac{ \delta \scrN(\vectau)[\vecphi,\vecpsi] }{ \| \Delta \vecf \|_{\R^n}^\alpha }
	\right\|_{L^p}
	&\leq
	C \| \vecphi^\prime \|_{W^{\sigma,2p} \cap L^\infty}
	\| \vecpsi^\prime \|_{W^{\sigma,2p} \cap L^\infty},
\end{align*}
and
$L^1$-estimates for
$\scrG_{(\alpha,p)}(\vecf)$
and
$\scrH_{(\alpha,p)}(\vecf)$.

Next,
we consider
$L^\infty$-estimates of
$\scrM_{(\alpha,p)}(\vecf)$,
$\scrG_{(\alpha,p)}(\vecf)$,
and
$\scrH_{(\alpha,p)}(\vecf)$
with the weight
$|\Delta s|^{\alpha-2\beta}$.
We assume that
$\vecf$,
$\vecphi$,
$\vecpsi \in C^{1,\beta}$
and
$s_1$,
$s_2 \in \R/\calL\Z$
satisfy
$\displaystyle 0<|\Delta s|<\calL/2$.
Because we have
\[
	1-x^\alpha \leq \left( \frac \alpha 2 +1 \right) (1-x^2)
\]
for all
$x \in [0,1]$,
it follows that
\begin{align}
	|\Delta s|^{\alpha-2\beta} \scrM_\alpha(\vecf)
	&\leq
	C \frac 1 { |\Delta s|^{2\beta} }
	\left( 1-\frac{ \| \Delta \vecf \|_{\R^n}^\alpha }{ |\Delta s|^\alpha } \right)
\nonumber\\
	&\leq
	C \frac 1 { |\Delta s|^{2\beta} }
	\left( 1-\frac{ \| \Delta \vecf \|_{\R^n}^2 }{ |\Delta s|^2 } \right)
\nonumber\\
	&\leq
	C \frac 1 { |\Delta s|^{2\beta+2} }
	\int_{s_2}^{s_1} \int_{s_2}^{s_1} \| \Delta_4^3 \vectau \|_{\R^n}^2 ds_3ds_4
\nonumber\\
	&\leq
	C.
\label{estimate D Malpha}
\end{align}
Moreover,
by Lemmas \ref{estimates N}--\ref{estimates inner product},
we have
\begin{align*}
	\left\| |\Delta s|^{\alpha-2\beta} \frac{ \scrR_1(\vecf)[\vecphi] }{ \| \Delta \vecf \|_{\R^n}^\alpha } \right\|_{L^\infty}
	&\leq
	C \| \vecphi^\prime \|_{L^\infty},
\\
	\left\| |\Delta s|^{\alpha-2\beta} \frac{ \scrR_2(\vecf)[\vecphi] }{ \| \Delta \vecf \|_{\R^n}^\alpha } \right\|_{L^\infty}
	&\leq
	C [\vecphi^\prime]_{C^{0,\beta}},
\\
	\left\| |\Delta s|^{\alpha-2\beta} \frac{ \scrS_1(\vecf)[\vecphi,\vecpsi] }
	{ \| \Delta \vecf \|_{\R^n}^\alpha } \right\|_{L^\infty}
	&\leq
	C \| \vecphi^\prime \|_{L^\infty} \| \vecpsi^\prime \|_{L^\infty},
\\
	\left\| |\Delta s|^{\alpha-2\beta} \frac{ \scrS_2(\vecf)[\vecphi,\vecpsi] }
	{ \| \Delta \vecf \|_{\R^n}^\alpha } \right\|_{L^\infty}
	&\leq
	C \{
	\| \vecphi^\prime \|_{L^\infty}
	( \| \vecpsi^\prime \|_{L^\infty} + [\vecpsi^\prime]_{C^{0,\beta}} )
	+
	\| \vecpsi^\prime \|_{L^\infty}
	( \| \vecphi^\prime \|_{L^\infty} + [\vecphi^\prime]_{C^{0,\beta}} )
	\},
\\
	\left\| |\Delta s|^{\alpha-2\beta} \frac{ \scrS_3(\vecf)[\vecphi,\vecpsi] }
	{ \| \Delta \vecf \|_{\R^n}^\alpha } \right\|_{L^\infty}
	&\leq
	C ( \| \vecpsi^\prime \|_{L^\infty} [\vecphi^\prime]_{C^{0,\beta}}
	+
	\| \vecphi^\prime \|_{L^\infty} [\vecpsi^\prime]_{C^{0,\beta}} ),
\\
	\left\| |\Delta s|^{\alpha-2\beta} \frac{ \scrS_4(\vecf)[\vecphi,\vecpsi] }
	{ \| \Delta \vecf \|_{\R^n}^\alpha } \right\|_{L^\infty}
	&\leq
	C [\vecphi^\prime]_{C^{0,\beta}} [\vecpsi^\prime]_{C^{0,\beta}},
\\
	\left\| |\Delta s|^{\alpha-2\beta} \frac{ \scrS_5(\vecf)[\vecphi,\vecpsi] }
	{ \| \Delta \vecf \|_{\R^n}^\alpha } \right\|_{L^\infty}
	&\leq
	C ( [\vecphi^\prime]_{C^{0,\beta}} + \| \vecphi^\prime \|_{L^\infty} )
	( [\vecpsi^\prime]_{C^{0,\beta}} + \| \vecpsi^\prime \|_{L^\infty} ).
\end{align*}
Therefore,
we obtain
\begin{align*}
	\left\|
	|\Delta s|^{\alpha-2\beta}
	\frac{ \delta \scrN(\vectau)[\vecphi] }{ \| \Delta \vecf \|_{\R^n}^\alpha }
	\right\|_{L^\infty}
	&\leq
	C \| \vecphi^\prime \|_{C^{0,\beta}}
\\
	\left\|
	|\Delta s|^{\alpha-2\beta}
	\frac{ \delta^2 \scrN(\vectau)[\vecphi,\vecpsi] }{ \| \Delta \vecf \|_{\R^n}^\alpha }
	\right\|_{L^\infty}
	&\leq
	C \| \vecphi^\prime \|_{C^{0,\beta}}
	\| \vecpsi^\prime \|_{C^{0,\beta}},
\end{align*}
and the
$L^\infty$-estimates for
$\scrM_{(\alpha,p)}(\vecf)$,
$\scrG_{(\alpha,p)}(\vecf)$,
and
$\scrH_{(\alpha,p)}(\vecf)$
with the weight
$|\Delta s|^{\alpha-2\beta}$.

Lastly,
we consider the continuity.
Because continuity on the outside of the diagonal set is clear,
we consider the property on the diagonal set.
First,
we assume
$\vecf$,
$\vecphi$,
$\vecpsi \in h^{1,\beta}$
with
$0 < \beta < 1$,
and
$\vecf$
is bi-Lipschitz.
Let
$R>0$
be sufficiently small.
In a similar manner to the proof of \eqref{estimate D Malpha},
it holds that
\[
	|\Delta s|^{\alpha-2\beta} \scrM_\alpha(\vecf)
	\leq
	\tilde{C} [\vectau]_{C^{0,\beta},R}
\]
for
$0<|\Delta s|<R$,
and therefore
\[
	\lim_{(s_1,s_2) \to (s,s)}
	|\Delta s|^{\alpha-2\beta} \scrM_\alpha(\vecf)
	= 0,
\]
where
$\tilde{C}$
is a positive constant depending only on
$\alpha$,
$p$,
and
$C_{\text b}$.
Moreover,
by Lemma \ref{estimates N},
we have
\begin{align*}
	\lim_{(s_1,s_2) \to (s,s)}
	\left|
	|\Delta s|^{\alpha-2\beta}
	\frac{ \scrR_i(\vecf)[\vecphi] }{ \| \Delta \vecf \|_{\R^n}^\alpha }
	\right|
	&=
	0,
	\quad
	(i=1,2),
\\
	\lim_{(s_1,s_2) \to (s,s)}
	\left|
	|\Delta s|^{\alpha-2\beta}
	\frac{ \scrS_i(\vecf)[\vecphi,\vecpsi] }{ \| \Delta \vecf \|_{\R^n}^\alpha }
	\right|
	&=
	0,
	\quad
	(i=1,\ldots,5).
\end{align*}
Considering Lemma \ref{estimates deltaDelta},
we obtain the desired continuity when
$\vecf$,
$\vecphi$,
$\vecpsi \in h^{1,\beta}$
with
$0<\beta<1$.

Next,
we consider the case where
$\beta=1$.
We denote the curvature vector of
$\vecf$
by
$\veckappa$,
i.e.\ $\veckappa = \vecf^{\prime\prime}$.
Using Lemma \ref{estimates N} and L'Hospital's theorem,
we have
\begin{align*}
	|\Delta s|^{\alpha-2} \scrM_\alpha(\vecf)
	&=
	|\Delta s|^{\alpha-2} \frac{ \vphi_\alpha(\scrN(\vectau)) }{ \| \Delta \vecf \|_{\R^n}^\alpha }
\\
	&=
	|\Delta s|^{\alpha-2} \frac{ \vphi_2(\scrN(\vectau)) }{ \| \Delta \vecf \|_{\R^n}^\alpha }
	\frac{ \vphi_\alpha(\scrN(\vectau)) }{ \vphi_2(\scrN(\vectau)) }
\\
	&=
	|\Delta s|^{\alpha-2} \frac{ \scrN(\vectau) }{ \| \Delta \vecf \|_{\R^n}^\alpha }
	\frac 1 { 1+\scrN(\vectau) }
	\frac{ \vphi_\alpha(\scrN(\vectau)) }{ \vphi_2(\scrN(\vectau)) }
\\
	&\to
	\frac \alpha {24} \| \veckappa(s) \|_{\R^n}^2
	\quad
	\text{as}\ (s_1,s_2) \to (s,s).
\end{align*}
Similarly,
it follows that
\begin{align*}
	|\Delta s|^{\alpha-2}
	\frac{ \scrR_1(\vecf)[\vecphi] }{ \| \Delta \vecf \|_{\R^n}^\alpha }
	&\to
	-\frac 1 6 \vectau(s) \cdot \vecphi^\prime(s) \| \veckappa(s) \|_{\R^n}^2,
\\
	|\Delta s|^{\alpha-2}
	\frac{ \scrR_2(\vecf)[\vecphi] }{ \| \Delta \vecf \|_{\R^n}^\alpha }
	&\to
	\frac 1 6 \veckappa(s) \cdot \vecphi^{\prime\prime}(s),
\\
	|\Delta s|^{\alpha-2}
	\frac{ \scrS_1(\vecf)[\vecphi,\vecpsi] }{ \| \Delta \vecf \|_{\R^n}^\alpha }
	&\to
	-\frac 1 6
	\{
	\vecphi^\prime(s) \cdot \vecpsi^\prime(s)
	-
	(\vectau(s) \cdot \vecphi^\prime(s)) (\vectau(s) \cdot \vecpsi^\prime(s))
	\}
	\| \veckappa(s) \|_{\R^n}^2,
\\
	|\Delta s|^{\alpha-2}
	\frac{ \scrS_2(\vecf)[\vecphi,\vecpsi] }{ \| \Delta \vecf \|_{\R^n}^\alpha }
	&\to
	\frac 1 6 ( \vectau(s) \cdot \vecphi^\prime(s) )
	( \veckappa(s) \cdot \vecpsi^{\prime\prime}(s)
	- \vectau(s) \cdot \vecpsi^\prime(s) \| \veckappa(s) \|_{\R^n}^2)
\\
	&\quad
	+
	\frac 1 6 ( \vectau(s) \cdot \vecpsi^\prime(s) )
	( \veckappa(s) \cdot \vecphi^{\prime\prime}(s)
	- \vectau(s) \cdot \vecphi^\prime(s) \| \veckappa(s) \|_{\R^n}^2),
\\
	|\Delta s|^{\alpha-2}
	\frac{ \scrS_3(\vecf)[\vecphi,\vecpsi] }{ \| \Delta \vecf \|_{\R^n}^\alpha }
	&\to
	- \frac 1 6
	(\vectau(s) \cdot \vecpsi^\prime(s)) (\veckappa(s) \cdot \vecphi^{\prime\prime}(s))
	- \frac 1 6
	(\vectau(s) \cdot \vecphi^\prime(s)) (\veckappa(s) \cdot \vecpsi^{\prime\prime}(s)),
\\
	|\Delta s|^{\alpha-2}
	\frac{ \scrS_4(\vecf)[\vecphi,\vecpsi] }{ \| \Delta \vecf \|_{\R^n}^\alpha }
	&\to
	\frac 1 6 \vecphi^{\prime\prime}(s) \cdot \vecpsi^{\prime\prime}(s),
\\
	|\Delta s|^{\alpha-2}
	\frac{ \scrS_5(\vecf)[\vecphi,\vecpsi] }{ \| \Delta \vecf \|_{\R^n}^\alpha }
	&\to
	- \frac 1 6
	\{ (\vectau(s) \cdot \vecphi^{\prime\prime}(s)) + (\vecphi^\prime(s) \cdot \veckappa(s)) \}
	\{ (\vectau(s) \cdot \vecpsi^{\prime\prime}(s)) + (\vecpsi^\prime(s) \cdot \veckappa(s)) \}
\end{align*}
as
$(s_1,s_2) \to (s,s)$.
Hence,
the desired continuity holds when
$\vecf$,
$\vecphi$,
$\vecpsi \in C^2$.

This completes
the proof of Theorem \ref{mainthm}.
\hfill
$\square$

%
%


\begin{thebibliography}{99}
\bibitem{ACFGH03}
	A.\ Abrams, J.\ Cantarella, J.\ H.\ G.\ Fu, M.\ Ghomi, and R.\ Howard, \textit{Circles minimize most knot energies}, Topology, \textbf{42} (2) (2003), 381--394.
\bibitem{B12-2}
	S.\ Blatt, \textit{The gradient flow of the M\"{o}bius energy near local minimizers}, Calc.\ Var.\ Partial Differential Equations, \textbf{43} (1) (2012), 1250010, 9 pp.
\bibitem{B12}
	S.\ Blatt, \textit{Boundedness and regularizing effects of O'Hara's knot energies}, J.\ Knot Theory Ramifications, \textbf{21} (2012), 1250010, 9pp.
\bibitem{B18}
	S.\ Blatt, \textit{The gradient flow of O'Hara's knot energies}, Math.\ Ann., \textbf{370} (3--4) (2018), 993--1061.
\bibitem{BRS16}
	S.\ Blatt, P.\ Reiter, and A.\ Schikorra, \textit{Harmonic analysis meets critical knots. Critical points of the M\"{o}bius energy are smooth}, Trans.\ Amer.\ Math.\ Soc.\ \textbf{368} (9) (2016), 6391--6438.
\bibitem{BR13}
	S.\ Blatt and P.\ Reiter, \textit{Stationary points of O'Hara's knot energies}, Manuscripta Math., \textbf{140} (1--2) (2013), 29--50.
\bibitem{BV19}
	S.\ Blatt and N.\ Vorderobermeier, \textit{On the analyticity of critical points of the M\"{o}bius energy}, Calc.\ Var.\ Partial Differential Equations, \textbf{58} (1) (2019), 28 pp.
\bibitem{FHW94}
	M.\ H.\ Freedman, Z.-X.\ He, and Z.\ Wang, \textit{M\"{o}bius energy of knots and unknots}, Ann.\ of Math., \textbf{139} (1994), 1--50.
\bibitem{He00}
	Z.-X.\ He, \textit{The Euler-Lagrange equation and heat flow for the M\"{o}bius energy}, Comm.\ Pure Appl.\ Math., \textbf{53} (4) (2000), 399--431. 
\bibitem{IN14}
	A.\ Ishizeki and T.\ Nagasawa, \textit{A decomposition theorem of the M\"{o}bius energy I: Decomposition and M\"{o}bius invariance}, Kodai.\ Math.\ J.,\ \textbf{37} (3) (2014), 737--754.
\bibitem{IN15}
	A.\ Ishizeki and T.\ Nagasawa, \textit{A decomposition theorem of the M\"{o}bius energy I\!I: Variational formulae and estimates}, Math.\ Ann., \textbf{363} (1--2) (2015), 617--635.
\bibitem{INpre}
	A.\ Ishizeki and T.\ Nagasawa, \textit{Decomposition of generalized O'Hara's energies}, preprint, arXiv:1904.06812 (2019).
\bibitem{O91}
	J.\ O'Hara, \textit{Energy of a knot}, Topology, \textbf{30} (2) (1991), 241--247.
\bibitem{O92}
	J.\ O'Hara, \textit{Family of energy functionals of knots}, Topology Appl.,\ \textbf{48} (2) (1992), 147--161.
\bibitem{O94}
	J.\ O'Hara, \textit{Energy functionals of knots I\!I}, Topology Appl.,\ \textbf{56} (1) (1994), 45--61.
\bibitem{Rei12}
	P.\ Reiter, \textit{Repulsive knot energies and pseudodifferential calculus for O'Hara's knot energy family $E^{(\alpha)}$, $\alpha \in [2,3)$}, Math.\ Nachr.,\ \textbf{285} (7) (2012), 889--913.
\end{thebibliography}
\end{document}